\def\update {03/04/2025}
 
\documentclass[11pt,leqno]{epiarticle}
\usepackage{epimath}
 \setpapertype{A4}

\usepackage[english]{babel} 

\usepackage[dvips]{graphicx}

\title{Schanuel Property \\ for Elliptic and Quasi--Elliptic Functions}
\titlemark{Schanuel Property \hfill \update}
\author{Michel Waldschmidt}
\authormark{M. Waldschmidt}
\date{\today} 
\journal{Hardy-Ramanujan Journal -- (2025), ---} 
\acceptation{submitted dd/mm/yyyy, accepted dd/mm/yyyy, revised \update}

\newcommand{\details}[1]{}

\renewcommand\atop[2]{\genfrac{}{}{0pt}{}{#1}{#2}}

\newcommand{\Z}{\mathbb{Z}}
\newcommand{\Q}{\mathbb{Q}} 
\def\R{\mathbb{R}}

\newcommand{\C}{\mathbb{C}}

\newcommand{\G}{\mathbb{G}}

\newcommand{\Real}{\Re{\mathrm{e}}}

\newcommand{\rme}{\mathrm {e}}
\newcommand{\rmi}{\mathrm {i}}

\newcommand{\End}{\mathrm{End}}

\newcommand{\oK}{\overline \oK}

\def\rmh{{\mathrm h}}

\def\rmH{{\mathrm H}}

\def\calU{{\mathcal{U}}}

\def\ux{\underline{x}}
\def\uX{\underline{X}}
\def\uY{\underline{Y}}
\def\uell{\underline{\ell}}
\def\um{\underline{m}}
\def\uup{{\underline{p}}}

\def\bfx{{\mathbf x}}
\def\bfz{{\mathbf z}}

\def\bfomega{{\boldsymbol{\omega}}}
\def\bfeta{{\boldsymbol{\eta}}}
\def\bflambda{{\boldsymbol{\lambda}}}

\def\rmd{{\mathrm d}}

\begin{document}

\overfullrule=0pt

\maketitle
 
\thanks{
\section*{Acknowledgement}
This text is an expanded version of the lecture the author gave in Chennai in December 2024 for the 
{\em International Conference on Number Theory and Related Topics ICNTRT-2024} at the The Institute of Mathematical Sciences (IMSc), Mathematics and Indian Institute of Technology IIT Madras. This conference was organized on the occasion of the 60th birthday of Srinivas Kotyada and Kalyan Chakraborty. The author is happy to felicitate his two friends and to thank the organizers of the conference. He is grateful to Cristiana Bertolin who initiated this investigation which gave rise also to our forthcoming joint papers \cite{BW1,BW2}. }

\begin{prelims}

\def\abstractname{Abstract}
\abstract{For almost all tuples $(x_1,\dots,x_n)$ of complex numbers, a strong version of Schanuel's Conjecture is true: {\em the $2n$ numbers $x_1,\dots,x_n, \rme^{x_1},\dots, \rme^{x_n}$ are algebraically independent}. Similar statements hold when one replaces the exponential function $\rme^z$ with algebraically independent functions. We give examples involving elliptic and quasi--elliptic functions, that we prove to be algebraically independent: $z$, $\wp(z)$, $\zeta(z)$, $\sigma(z)$, exponential functions, and Serre functions related with integrals of the third kind. }


\MSCclass{11J81, 11J89, 14K25}
 
\tableofcontents

\end{prelims}

\section{Extended abstract}
Schanuel's Conjecture is one of the main open problems in the theory of transcendental numbers. We remark that a stronger statement, which we call {\em Strong Schanuel Property}, is valid for almost all numbers, and we give an explicit example where the Strong Schanuel Property is satisfied. 

The algebraic independence statement for almost all tuples is valid not only for the usual exponential function, but more generally for algebraically independent functions. We develop this remark for functions related with elliptic curves; we use the Weierstrass model and investigate the algebraic independence of values of elliptic $\wp$ functions as well as of quasi--elliptic functions $\zeta$ and $\sigma$. We also consider values of the functions which were introduced by Serre \cite[Appendix]{Wald79a} in order to study elliptic integrals of the third kind. 

The result that we obtain for almost all numbers is a first step before suggesting conjectures extending Schanuel's one to these elliptic and quasi--elliptic functions. The above functions are involved in the parametrization of the exponential of commutative algebraic groups. 
In \cite{BW1,BW2}, we show how this approach contributes to produce new conjectures having a geometrical origin, which means that they are special cases of the Grothendieck-Andr\'{e} generalized period Conjecture, involving motives, like in earlier works of Cristiana Bertolin (see in particular \cite{B02} and \cite{B20}).

In the appendix, motivated by a joke of D.W. Masser, we state a result involving both a Weierstrass zeta function and the Riemann zeta function.

\section{Introduction}

The following statement was proposed as a conjecture by Stephen Schanuel while he attended a course by Serge Lang, which gave rise to the book \cite{Lang}.

\medskip
\noindent
{\bf Schanuel's Conjecture}. {\em
If $x_1,\dots,x_n$ are complex numbers linearly independent over $\Q $, then at least $n$ of the numbers $x_1,\dots,x_n, \rme^{x_1},\dots, \rme^{x_n}$ are algebraically independent over $\Q$.
}
 
\medskip
 In other terms the transcendence degree over $\Q$ of the field $\Q(x_1,\dots,x_n, \rme^{x_1},\dots, \rme^{x_n})$ is at least $n$. 

This lower bound $n$ for the transcendence degree is clearly best possible: for instance if $x_1,\dots,x_n$ are algebraic, then the transcendence degree is $n$. This Theorem of Lindemann--Weierstrass is one of the very few known special cases where Schanuel's Conjecture is known to hold. 
 
 Our first remark is that a stronger statement is valid for almost all tuples (for Lebesgue measure in $\C^n$). 
 
\vskip 0.2truecm
\par\noindent \textbf{Definition.}
\textit{ 
We say that an $n$--tuple $(x_1,\dots,x_n)$ of complex numbers satisfies the {\em Strong Schanuel Property} if the $2n$ numbers 
$x_1,\dots,x_n,\rme^{x_1},\dots,\rme^{x_n}$ are algebraically independent over $\Q$.}

In Example \ref{example}, \S~\ref{S:ExampleStrongSchanuelExplicit}, we produce uncountably many explicit $n$--tuples $(x_1,\dots,x_n)$ of complex numbers that satisfy the Strong Schanuel Property.

The fact that the Strong Schanuel Property holds for almost all tuples of complex numbers merely uses the fact that the exponential function is transcendental (over $\C(z)$). It is the special case $K=\Q$, $t=2$, $f_1(z)=z$, $f_2(z)=\rme^z$ of the following result:

\begin{proposition}\label{Prop:SSPIF}
	Let $K$ be a finitely generated extension of $\Q$. Let $f_1,\dots,f_t$ be meromorphic functions in $\C$ which are algebraically independent over $K$. 
	Then for almost all tuples $(z_1,\dots,z_n)$ of complex numbers, the $nt$ numbers 
	\begin{equation}\label{Equation:SSPIF}
f_j(z_i) \qquad (i=1,\dots,n,\quad j=1,\dots,t)
	\end{equation}
	are algebraically independent over $K$. 
\end{proposition}

The converse is plain: if $f_1,\dots,f_t$ are meromorphic functions in $\C$ which are algebraically dependent over $\C$, if we denote by $K$ the field generated over $\Q$ by the coefficients of a polynomial in $\C[X_1,\dots,X_t]$ which vanishes when the variables are specialized as $f_1,\dots,f_t$, then for all $n$ and all tuples $(z_1,\dots,z_n)$ of complex numbers the $nt$ numbers \eqref{Equation:SSPIF} are algebraically dependent over $K$.

 Our main result on algebraic independence of periodic and quasi--periodic functions is the following theorem, which involves Weierstrass $\wp$, $\zeta$ and $\sigma$ functions as well as Serre functions $f_u$ (see \S \ref{S:functions}). Among a number of special cases it contains \cite[Lemma 5]{Reyssat82} and \cite[Lemma 20.10]{M}. We identify the {\em ring $\End(E)$ of endomorphisms of $E$} with the set of complex numbers $\alpha$ such that $\alpha\Omega\subset\Omega$, where $\Omega$ is the lattice of periods of $E$. 
 The field of fraction $k:=\End(E)\otimes_\Z\Q$ of $\End(E)$ is the {\em field of endomorphisms of $E$}, namely $k=\Q$ in the non--CM case, and $k=\Q(\tau)$ in the CM case, whenever $\tau$ is the quotient of two independent periods. 

\begin{theorem}\label{Theorem:AlgebraicIndependenceFunctions}
	Let $t_1,\dots,t_r$ be complex numbers linearly independent over $\Q$. Let $u_1,\dots,u_s$ be complex numbers. 
	\\
1. 	 Assume that $\omega_1,\omega_2,u_1,\dots,u_s$ are linearly independent over $\Q$.
	Then the $4+r+s$ functions 
\[ 
	z, \wp(z), \zeta(z),\sigma(z), \rme^{t_1z},\dots, \rme^{t_rz}, f_{u_1}(z),\dots,f_{u_s}(z)
\]
	are algebraically independent. 
\\
2. Assume further that the elliptic curve $E$ has complex multiplication with field of endomorphisms $k$. Let $\alpha_1,\dots,\alpha_s$ be elements in $\End(E)\smallsetminus\Z$. Assume that $\omega_1,u_1,\dots,u_s$ are linearly independent over $k$. Then the $4+r+2s$ functions 
\[ 
	z, \wp(z), \zeta(z),\sigma(z), \rme^{t_1z},\dots, \rme^{t_rz}, f_{u_1}(z),\dots,f_{u_s}(z), f_{u_1}(\alpha_1z),\dots,f_{u_s}(\alpha_sz)
\]
	are algebraically independent. 
\end{theorem}

Proposition \ref{Proposition:AlgebraicIndependenceSerreFunctions} shows that in item 1 the assumption that $\omega_1,\omega_2,u_1,\dots,u_s$ are linearly independent over $\Q$ cannot be omitted.

From Proposition \ref{Prop:SSPIF} and Theorem \ref{Theorem:AlgebraicIndependenceFunctions} we deduce at once:

\begin{corollary}[Strong elliptic Schanuel property]\label{Cor:SESP}
Let $K$ be a finitely generated extension of $\Q$. 
Let $t_1,\dots,t_r$ be complex numbers linearly independent over $\Q$. Let $u_1,\dots,u_s$ be complex numbers.
\\
1. Assume that $\omega_1,\omega_2,u_1,\dots,u_s$ are linearly independent over $\Q$. Then for almost all $n$--tuples $(z_1,\dots,z_n)$ of complex numbers, the $n(4+r+s)$ numbers
\[ 
	z_i, \wp(z_i), \zeta(z_i), \sigma(z_i), \rme^{t_1z_i},\dots, \rme^{t_rz_i}, f_{u_1}(z_i),\dots,f_{u_s}(z_i)
	\quad (i=1,\dots,n) 	
\] 
are algebraically independent over $K(g_2,g_3)$. 
\\
2. Assume that the elliptic curve $E$ has complex multiplication; let $k$ be the field of endomorphisms of $E$. Let $\alpha_1,\dots,\alpha_s$ be elements in $\End(E)\smallsetminus\Z$. Assume that $\omega_1,u_1,\dots,u_s$ are linearly independent over $k$. Then for almost all $n$--tuples $(z_1,\dots,z_n)$ of complex numbers, the $n(4+r+2s)$ numbers
\[ 
	z_i, \wp(z_i), \zeta(z_i), \sigma(z_i), \rme^{t_1z_i},\dots, \rme^{t_rz_i}, f_{u_1}(z_i),\dots,f_{u_s}(z_i), f_{u_1}(\alpha_1z_i),\dots,f_{u_s}(\alpha_sz_i)
	\quad (i=1,\dots,n) 
\] 
are algebraically independent over $K(g_2,g_3)$. 
\end{corollary}

It is a challenge to state a conjecture, extending Schanuel's one, which would hold for all tuples $(z_1,\dots,z_n)$ of complex numbers. Such a statement should hold in particular when these $n$ numbers are algebraic, like in the Lindemann--Weierstrass Theorem. Hence, in the conclusion, one should not predict anything better than a transcendence degree $\geqslant n(3+r+s)$ for part 1 (resp. $\geqslant n(3+r+2s)$ in the CM case for part 2). Such a lower bound is sometimes too optimistic (see for instance the so--called {\em exceptional case} in \cite{BW1}).
Also, we would need to assume some condition of linear independence of the $z_i$ over the field of endomorphisms of $\wp$. One should as well replace the finitely generated field $K$ by $\Q\bigl(g_2,g_3,\wp(u_1),\dots,\wp(u_n)\bigr)$.
 
 The elliptic function $\wp$ parametrizes the exponential of an elliptic curve, the zeta function $\zeta$ occurs in the parametrization of the exponential of an extension of an elliptic curve by the additive group, Serre functions $f_u$ occur in the parametrization of the exponential of extensions of an elliptic curve by the multiplicative group \cite[Appendix by Serre]{Wald79a}, 
 \cite[Chapter 20, Exercise 20.104]{M}.
 
 In forthcoming joint papers \cite{BW1,BW2} with Cristiana Bertolin, we propose new conjectures in the style of Corollary \ref{Cor:SESP} but which we expect to be true for all tuples of complex numbers. 

\section{Explicit examples of tuples satisfying the Strong Schanuel Property}\label{S:ExampleStrongSchanuelExplicit}
We now give explicit examples of tuples satisfying the Strong Schanuel Property. 
 We denote by $\Vert\cdot\Vert$ the distance to the nearest integer. 

\begin{theorem}\label{StrongSchanuelExplicit}
	Let $\psi:\Z_{>0}\to \R_{>0}$ be a decreasing function such that 
\[ 
	\psi(q)<\rme^{-q^4}.
\] 
	Let $x_1,\dots,x_n$ be real numbers. Assume that there exists a sequence $(q_k)_{k\geqslant 0}$ of positive integers such that
\[ 
	0<k^{n-1}\Vert q_kx_n\Vert\leqslant 
	k^{n-2}\Vert q_k x_{n-1}\Vert\leqslant 
	\cdots
	\leqslant
	k \Vert q_k x_2\Vert\leqslant \Vert q_k x_1\Vert\leqslant \psi(q_k)
\] 
	for all $k\geqslant 0$. 
	Then the $n$--tuple $(x_1,\dots,x_n)$ satisfies the Strong Schanuel Property. 
\end{theorem}

Before proving Theorem \ref{StrongSchanuelExplicit}, we show how it produces uncountably many explicit examples of $n$--tuples of real numbers (in fact Liouville numbers) satisfying the Strong Schanuel Property.

\begin{example}\label{example}
Define a sequence $(q_k)_{k\geqslant 0}$ of positive integers by $q_0=1$ and $q_{k+1}=3^{q_k^4}$ for $k\geqslant 0$.
 For $\ell\geqslant 1$ and $1\leqslant i\leqslant n$, let $\epsilon_\ell^{(i)}\in\{-1,+1\}$. For $i=1,\dots,n$ and $k\geqslant 1$, set
\[
x_i=\sum_{\ell=1}^\infty \epsilon_\ell^{(i)} 
\frac{\bigl(4(\ell-1)\bigr)^{n-i}} {q_\ell}
\qquad \text{and}\qquad
p_k^{(i)}=q_k\sum_{\ell = 1}^k \epsilon_\ell^{(i)}\frac{\bigl(4(\ell-1)\bigr)^{n-i}} {q_\ell} ,
\]
so that 
\[
q_kx_i-p_k^{(i)}=q_k\sum_{\ell=k+1}^\infty \epsilon_\ell^{(i)} 
\frac{\bigl(4(\ell-1)\bigr)^{n-i}} {q_\ell}
=\epsilon_{k+1}^{(i)} 
\frac{(4 k)^{n-i} q_k} {q_{k+1}}
+q_k\sum_{\ell=k+2}^\infty \epsilon_\ell^{(i)} 
\frac{\bigl(4(\ell-1)\bigr)^{n-i}} {q_\ell}
\cdotp
\]
Then for $1\leqslant i\leqslant n$ and $k$ sufficiently large, we have
\[
\frac{(4k)^{n-i} q_k} {2q_{k+1}} \leqslant 
\left| q_k x_i- p_k^{(i)} \right|\leqslant \frac{ 2 (4k)^{n-i}q_k} {q_{k+1}} \cdotp
\]
Hence for sufficiently large $k$ we have 
\[
\Vert q_k x_i\Vert \leqslant 
\frac{ 2 (4k)^{n-i}q_k} {q_{k+1}} 
=
\frac{(4k)^{n-i+1} q_k} {2kq_{k+1}} \leqslant 
 \frac 1 k\Vert q_k x_{i-1}\Vert 
\]
for $2\leqslant i\leqslant n$ and 
\[
\Vert q_k x_1\Vert \leqslant \frac{ 2 (4k)^{n-1}q_k} {q_{k+1}}\leqslant \rme^{-q_k^4}.
\]
\end{example}
From Theorem \ref{StrongSchanuelExplicit}, it follows that the numbers $x_1,\dots,x_n,\rme^{x_1},\dots,\rme^{x_n}$ are algebraically independent over $\Q$.

The proof of Theorem \ref{StrongSchanuelExplicit} rests on three preliminary lemmas. The first one is a special case of a Durand's result \cite{D}, quoted in \cite[Theorem p.~226]{W1}.

\begin{lemma}[A.~Durand]\label{Lemme:Durand}
	Let $x_1,\dots,x_n$ be real numbers with $n\geqslant 1$. Assume that for any integer $k\geqslant 1$ there exists an integer $q$ such that 
\[ 
	0<k^{n-1}\Vert qx_n\Vert\leqslant 
	k^{n-2}\Vert qx_{n-1}\Vert\leqslant 
	\cdots
	\leqslant 
	k \Vert qx_2\Vert\leqslant \Vert qx_1\Vert\leqslant q^{-k}.
\] 
	Then $x_1,\dots,x_n$ are algebraically independent. 
\end{lemma}

\medskip
The next auxiliary lemma is a lower bound for a non homogeneous linear form in one logarithm; there are several such estimates, we use a special case of \cite[Theorem 9.1]{W2}. We denote by $\rmh$ the absolute logarithmic height of an algebraic number.

\begin{lemma}\label{Lemme:fl1log}
	Let $\alpha$ be a nonzero algebraic number, $\log\alpha$ a logarithm of $\alpha$, $\beta$ a non zero algebraic number, $D$ the degree of the field $\Q(\alpha,\beta)$, and $A$, $B$ real numbers satisfying 
\[ 
	\log A\geqslant \max\{1, \rmh(\alpha), |\log \alpha|\},\qquad
	B\geqslant \max\{\rme, \rmh(\beta), D\log A\}.
\] 
	Then
\[ 
	|\beta-\log\alpha|\geqslant \exp\{-2^{26}D^3\log A\log B\}.
\] 
\end{lemma}

\medskip
We will need also an upper bound for the distance of a complex number $\theta$ to the nearest zero of a polynomial $F\in\Z[X]$ in terms of $|F(\theta)|$. The next lemma is due to N.I. Fel'dman \cite[Lemma 15.13]{W2}. Here, $\rmH$ denotes the usual height of a polynomial, i.e. the maximum of the absolute values of its coefficients.

\begin{lemma}\label{Lemme:Feldman}
	Let $F$ be a nonzero polynomial in $\Z[X]$ of degree $D$. Let $\theta$ be a complex number, $\alpha$ a root of $F$ at minimal distance of $\theta$ and $\ell$ the multiplicity of $\alpha$ as a root of $F$. Then
\[ 
	|\theta-\alpha|^\ell \leqslant D^{3D-2} \rmH(F)^{2D}|F(\theta)|.
\] 
\end{lemma}

\medskip

We can now prove the main result of this section.

\begin{proof}[Proof of Theorem \ref{StrongSchanuelExplicit}]
	Under the assumptions of Theorem \ref{StrongSchanuelExplicit}, assume that there exists a nonzero polynomial $f\in\Z[X_1,\dots,X_n,Y_1,\dots,Y_n]$ such that
\[ 
	f(x_1,\dots,x_n,\rme^{x_1},\dots,\rme^{x_n})=0.
\] 
	Let $L$ be the maximum of the partial degrees of $f$ with respect to the $2n$ variables. 
	Write
\[ 
	f(\uX,\uY)=
	\sum_{\uell}
	\sum_{\um}
	f_{\uell,\um}\uX^{\uell}
	\uY^{\um}
\] 
	where $\uell$ stands for $(\ell_1,\dots,\ell_n)$ with $0\leqslant \ell_i\leqslant L$ ($1\leqslant i\leqslant n)$, 
	$\um$ for $(m_1,\dots,m_n)$ with $0\leqslant m_j\leqslant L$ ($1\leqslant j\leqslant n)$, 
	$\uX^{\uell}$ for $X_1^{\ell_1}\cdots X_n^{\ell_n}$
	and 
	$\uY^{\um}$ for $Y_1^{m_1}\cdots Y_n^{m_n}$.
	
Using the assumption of Theorem \ref{StrongSchanuelExplicit} on the existence of a sequence $(q_k)_{k\geqslant 0}$, we introduce 
$n$ sequences $(p_{k1})_{k\geqslant 1},\dots,(p_{kn})_{k\geqslant 1}$ of integers such that, for $k\geqslant 1$ and $1\leqslant i\leqslant n$,
\[ 
	|q_kx_i-p_{ki}| =\Vert q_kx_i\Vert.
\] 
Lemma \ref{Lemme:Durand} implies that the numbers $x_1,\ldots,x_n$ are algebraically independent.
 	
	We substitute $p_{ki}/q_k$ to $X_i$ and $\rme^{p_{ki}/q_k}$ to $Y_i$ in $f(\uX,\uY)$: set
\[ 
	\xi_k:=f(\uup_k/q_k,\rme^{\uup_k/q_k})=
	\sum_{\uell}
	\sum_{\um}
	f_{\uell,\um}(\uup_k/q_k)^{\uell}
	(\rme^{\uup_k/q_k})^{\um},
\] 
	where $\uup_k/q_k$ stands for $(p_{k1}/q_k,\dots,p_{kn}/q_k)$ 
	and 
	$\rme^{\uup_k/q_k}$ for $(\rme^{p_{k1}/q_k},\dots,\rme^{p_{kn}/q_k})$. Then
\[
\xi_k=
	\sum_{\uell}
	\sum_{\um}
	f_{\uell,\um}
	\left((\uup_k/q_k)^{\uell}
	(\rme^{\uup_k/q_k})^{\um}-\ux^{\uell}(\rme^{\ux})^{\um}
	\right)
\]	
where $\ux^{\uell}=x_1^{\ell_1}\cdots x_n^{\ell_n} $ and $(\rme^{\ux})^{\um}=\rme^{m_1x_1+\cdots+m_nx_n}$. One deduces 
\begin{equation}\label{Equation:|xii_ell}
|\xi_k|\leqslant c_1 \psi(q_k) 
\end{equation} 
with a constant $c_1$ independent on $k$. 
	The number $\xi_k$ can be written as $q_k^{-nL}F_k(\rme^{1/q_k})$ where $F_k\in\Z[T]$ is the polynomial 
\[ 
	F_k(T)=\sum_{\uell}
	\sum_{\um}
	f_{\uell,\um}\uup_k^{\uell} q_k^{nL-\ell_1-\cdots-\ell_n}
	T^{p_{k1} m_1+\cdots+p_{kn} m_n}.
\] 
	We claim that, for sufficiently large $k$, the exponents $p_{k1} m_1+\cdots+p_{kn} m_n$ ($0\leqslant m_i\leqslant L$) are pairwise distinct. Indeed, otherwise, since the set of $(m_1,\dots,m_n)$ has at most $(L+1)^n$ elements (independent on $k$), we would have a non trivial linear relation 
\[ 
	a_1 p_{k1} +\cdots+a_np_{kn} =0
\] 
	with $-L\leqslant a_i\leqslant L$, ($1\leqslant i\leqslant n$) valid for infinitely many $k$. Dividing by $q_k$ and taking the limit of the left hand side as $k\to\infty$ we would deduce 
\[ 
	a_1 x_1 +\cdots+a_nx_n =0, 
\] 
	which is not true since $x_1,\ldots,x_n$ are algebraically independent, hence linearly independent.

	We claim that, for sufficiently large $k$, the polynomial $ F_k(T)$ is not zero. Otherwise, we would have 
\[ 
	\sum_{\uell} 
	f_{\uell,\um}
	p_{k1}^{\ell_1}\cdots p_{kn}^{\ell_n}=0
\] 
	for all $(m_1,\dots,m_n)$ with $0\leqslant m_i\leqslant L$. Again, by taking the limit as $k\to\infty$, the relations
\[ 
	\sum_{\uell} 
	f_{\uell,\um}x_1^{\ell_1}\cdots x_n^{\ell_n}=0
\] 
	would give a contradiction with the algebraic independence of $x_1,\ldots,x_n$.
	
	Let $\alpha_k$ be a root of $ F_k(T)$ at a nearest distance of $\rme^{1/q_k}$. The degree of $\alpha_k$ is at most $c_2q_k$ and its logarithmic height at most $c_3\log q_k$, with positive constants $c_2$ and $c_3$ which do not depend on $k$. Using \eqref{Equation:|xii_ell} together with Lemma \ref{Lemme:Feldman} we deduce 
\[ 
\left| \rme^{1 /q_k}-\alpha_k\right|\leqslant c_4\psi(q_k),
\]
which implies that $\alpha$ has a logarithm which satisfies 
\begin{equation}\label{Equation:upperbound}
\left|\frac 1 {q_k}-\log \alpha_k\right|\leqslant c_5\psi(q_k).
\end{equation}
On the other hand, Lemma \ref{Lemme:fl1log} with $D=c_2q_k$, $A=q_k^{c_3}$, $B=q_k$, yields 
\begin{equation}\label{Equation:lowerbound}
	\left|\frac 1 {q_k}-\log \alpha_k\right|\geqslant \exp\{-c_5 q_k^3(\log q_k)^2\}.
\end{equation}
	The estimates \eqref{Equation:upperbound} and \eqref{Equation:lowerbound} are not compatible, hence the result. 
 \hfill $\square$
 \end{proof}

\section{Elliptic and quasi--elliptic functions}\label{S:functions}

Let $\Omega = \Z \omega_1 + \Z \omega_2$ be a lattice in $\C$ with elliptic invariants $g_2,g_3$. 
We consider the following Weierstrass functions (see for instance \cite{M}): the Weierstrass sigma function
\begin{equation}\label{Equation:sigmaDef}
\sigma(z)=z\prod_{\omega\in\Omega\smallsetminus\{0\}}\left(1-\frac z \omega\right) \exp\left( \frac z \omega + \frac {z^2}{2\omega^2}\right),
\end{equation}
the Weierstrass zeta function $\zeta=\sigma'/\sigma$ and the Weierstrass elliptic function $\wp=-\zeta' $.
 
Further, for $u\in\C\smallsetminus\Omega$, we introduce the Serre function
\[
f_u(z) = \frac{\sigma(z+u)}{\sigma(z)\sigma(u)} \rme^{-\zeta(u) z }.
\]
The periods of the Weierstrass elliptic function $\wp$ are elliptic integrals of the first kind. 
 The Weierstrass zeta function $\zeta$ has quasi--periods $\eta=\zeta(z+\omega)-\zeta(z)$ which are given by elliptic integrals of the second kind. 
Serre functions $f_u$ were introduced in an appendix of \cite{Wald79a} to investigate elliptic integrals of the third kind. 

For $\omega\in\Omega$, we have
$\wp(z+\omega)=\wp(z)$,
 $\zeta(z+\omega)=\zeta(z)+\eta(\omega) $ and 
\begin{equation}\label{Equation:sigma} 
\sigma(z+\omega)=\epsilon(\omega)\sigma(z)\exp\left(\eta(\omega)\left(z+\frac \omega 2\right)\right)
\end{equation}
with 
\[ 
\epsilon(\omega)=
\begin{cases}1
&
\hbox{if $\omega/2\in\Omega$},
\\ -1
&
\hbox{ if $\omega/2\not\in\Omega$}.
\end{cases}
\] 
Further, 
\[
f_u(z+\omega)= f_u(z)\exp\left(\eta(\omega)u-\omega\zeta(u)\right).
\]

The quasi--periodicity of the Weierstrass zeta function defines 
 a $\Z$--linear map (homomorphism of abelian groups) 
 \[
 \begin{matrix}
 \eta:& \Omega& \to& \C
 \\
 & \omega &\mapsto & \eta(\omega) .
 \end{matrix}
 \] 
 We set $\eta_i= \eta(\omega_i)$ for $i=1,2$. We will choose the order of $\omega_1,\omega_2$ with positive imaginary part of $\omega_2/\omega_1$, so that Legendre relation \cite[Exercise 20.33]{M} is 
 \begin{equation}\label{Equation:Legendre}
 \omega_2\eta_1-\omega_1\eta_2=2\pi\rmi.
 \end{equation}
 This relation holds for a basis $(\omega_1,\omega_2)$, while for a pair of two linearly independent periods $\omega'_1,\omega'_2$ we have
\begin{equation}\label{Equation:LegendreGeneralise}
 \omega'_2\eta(\omega'_1)-\omega'_1\eta(\omega'_2)\in 2\pi\rmi\Z\smallsetminus\{0\}.
\end{equation}
For $u\in\C\smallsetminus\Omega$ and $\omega\in\Omega$, we define 
\begin{equation}\label{Equation:lambdaDef}
\lambda(u,\omega)=\eta(\omega)u-\omega\zeta(u),
\end{equation}
so that the quasi--periodicity relation of $f_u$ is
\begin{equation}\label{Equation:lambda}
f_u(z+\omega)= f_u(z)\exp\bigl(\lambda(u,\omega)\bigr) \qquad (\omega\in\Omega).
\end{equation}
From Legendre relation \ref{Equation:Legendre} we deduce
\begin{equation}\label{Equation:ConsequenceLegendre}
	\lambda(u,\omega_1)\omega_2-\lambda(u,\omega_2)\omega_1=(\omega_2\eta_1-\omega_1\eta_2)u=2\pi\rmi u.
\end{equation}

\medskip

\section{Strong Schanuel Property for values of algebraically independent functions}

\begin{proof}[Proof of Proposition \ref{Prop:SSPIF}]
We will use the following facts.
\\
$\bullet$	If $F$ is a nonzero meromorphic function in $\C^n$, then the set $Z(F)$ of zeroes of $F$ in $\C^n$ has Lebesgue measure zero. 
\\
$\bullet$	A countable union of sets of Lebesgue measure zero has Lebesgue measure zero.
\\
$\bullet$	The set of polynomials in $nt$ variables with coefficients in $K$ is countable. 	
	
	For $P$ a nonzero polynomial in $nt$ variables with coefficients in $K$, define a nonzero meromorphic function in $\C^n$ by 
\[
F_P(z_1,\dots,z_n)=P\bigl( (f_j(z_i) )_{\atop{1\leqslant i\leqslant n}{1\leqslant j\leqslant t}}\bigr)
\]
and let $Z(F_P)\subset \C^n$ be the set of zeroes of $F_P$. The set of tuples $(z_1,\dots,z_n)\in \C^n$ such that the $nt$ numbers 
\eqref{Equation:SSPIF} 	are algebraically dependent over $K$ is the union of all $Z(F_P)$ with $P\in K[(X_{ij})_{
	\atop{1\leqslant i\leqslant n}{1\leqslant j\leqslant t}}]\smallsetminus\{0\}$. Hence the result. 
 \hfill $\square$
 \end{proof}

\section{Algebraic independence of periodic and quasi--periodic functions}

In this section we prove Theorem \ref{Theorem:AlgebraicIndependenceFunctions}. We need preliminary results.

\subsection{Auxiliary results}

In a multiplicative $\Z$--module, which is nothing else than an abelian group written multiplicatively, linear dependence is called {\em multiplicative dependence}. For instance 
the set $(\C^\times)^\Omega$ of sequences of nonzero complex numbers indexed by $\Omega$ is a multiplicative $\Z$--module; 
elements $\bfx_1,\dots,\bfx_m$ of $(\C^\times)^\Omega$, where $\bfx_i=(x_{i,\omega})_{\omega\in\Omega}$, are multiplicatively dependent if there exists $(h_1,\dots,h_m)\in\Z^m\smallsetminus\{0\}$ such that 
\[
\bfx_1^{ h_1}\cdots \bfx_m^{h_m}=1,
\]
that is
\[
x_{1,\omega}^{ h_1}\cdots x_{m,\omega}^{h_m}=1 \quad \hbox{ for all } \quad \omega\in\Omega.
\]
For $\bfz=(z_\omega)_{\omega\in\Omega}\in \C^\Omega $, $\exp(\bfz)\in(\C^\times)^\Omega$ denotes $\bigl(\exp(z_\omega)\bigr)_{\omega\in\Omega}$. 
For $\bfz_1,\dots,\bfz_m$ in $ \C^\Omega $, the elements 
 $\exp(\bfz_1),\dots,\exp(\bfz_m)$ are multiplicatively dependent in $(\C^\times)^\Omega$ if and only if there exists $(h_1,\dots,h_m)\in\Z^m\smallsetminus\{0\}$ such that 
\[
h_1z_{1,\omega}+\cdots +h_mz_{m,\omega}\in 2\pi \rmi \Z \quad \hbox{ for all } \quad \omega\in\Omega.
\]
In this case, let 
\[
h_0(\omega)=\frac 1 {2\pi \rmi} \left(h_1 z_{1,\omega}+\cdots +h_m z_{m,\omega}\right).
\]
Then, for $(a,b)\in\Z^2$, 
\[
h_0(a\omega_1+b\omega_2)=ah_0(\omega_1)+bh_0(\omega_2).
\]

We introduce two elements of $\C^\Omega$:
\[
\bfomega=(\omega)_{\omega\in\Omega}, \quad 
\bfeta =(\eta(\omega) )_{\omega\in\Omega}.
\]
In the $\C$--vector space $\C^\Omega$, for $t\in \C$ and $\bfx=(x_\omega)_{\omega\in\Omega}\in\C^\Omega$, we have 
\[
t \bfx=\bfx t=(tx_\omega)_{\omega\in\Omega}\in\C^\Omega.
\] 
Based on \eqref{Equation:lambdaDef}, for $u \in \C \smallsetminus \Omega$, we define an element 
in $\C^\Omega$ as 
\[
\bflambda(u,\bfomega)= \bfeta u-\bfomega\zeta(u)=\bigl(\lambda(u,\omega)\bigr)_{\omega\in\Omega}
=\bigl(\eta(\omega) u- \omega\zeta(u)\bigr)_{\omega\in\Omega}.
\] 

Recall that for complex numbers $t_1,\dots,t_r$, the $r$ numbers $\rme^{t_1},\dots,\rme^{t_r}$ are multiplicatively independent if and only if the 
$r+1$ numbers $2\pi \rmi, t_1,\dots,t_r$ are linearly independent over $\Q$. 

Our first auxiliary result is a generalisation of Vandermonde determinants.

 \begin{lemma}\label{Lemma:Vandermonde}
 Let $(r_t)_{t\geqslant 0}$ be a sequence of monic polynomials in $\C[T]$ where the degree of $r_t$ is $t$. 
Let $t_1,\dots,t_m$ positive integers and $w_1,\dots,w_m$ complex numbers. Set $D=t_1+\cdots+t_m$. Then the determinant of the $D\times D$ matrix 
\[
M:=\bigl(
r_t(a)w_j^a\bigr)_{\atop{0\leqslant a< D}{0\leqslant t<t_j,\; 1\leqslant j\leqslant m}}
\]
is
\begin{equation}\label{Equation:determinant}
k(t_1,\dots,t_m) \prod_{j=1}^m w_i^{t_j(t_j-1)/2}
\prod_{1\leqslant i<j\leqslant m}(w_j-w_i)^{ t_it_j}
\end{equation}
with
\[
k(t_1,\dots,t_m)=\prod_{j=1}^m\prod_{t=1}^{t_j-1} t!.
\]
\end{lemma}

The matrix $M$ is 
\[
M=\bigl(M_1(w_1) \; M_2(w_2) \; \dots M_m(w_m) \bigr)
\]
where, for $1\leqslant j\leqslant m$ and $w\in\C$, $M_j(w)$ is the $D\times t_j$ matrix
 \[
M_j(w):=\bigl(
r_t(a)w_j^a\bigr)_{\atop{0\leqslant a< D}{ 0\leqslant t<t_j}},
\]
namely
\[
M_j(w):=
\begin{pmatrix}
r_0(0)&r_1(0)&r_2(0)&\cdots&r_{t_j-1}(0)
\\
r_0(1)w&r_1(1)w&r_2(1) w&\cdots&r_{t_j-1}(1)w
\\ 
r_0(2)w^2&r_1(2)w^2&r_2(2)w^2&\cdots&r_{t_j-1}(2)w^2
\\
\vdots&\vdots&\vdots&\ddots&\vdots
\\
r_0(D-1)w^{D-1}&r_1(D-1)w^{D-1}&r_2(D-1)w^{D-1}&\cdots&r_{t_j-1}(D-1)w^{D-1}
\end{pmatrix}.
\] 

\begin{proof}
By means of linear combinations of the columns, we see that the determinant does not depend on the polynomials $r_t$. Without loss of generality we assume $r_t(z)=z^t$.

For $a\geqslant 0$ and $t\geqslant 0$ we have 
\[
\left(z \frac \rmd{\rmd z}\right)^t z^a=a^tz^a.
\]
For $a=0$ we have $z^a=1$ and $(z\rmd/\rmd z)^t 1=\delta_{t,0}$ (Kronecker symbol; in particular $z^a=1$ for $z=a=0$). 

Denote by $\C[z]_{<D}$ the $D$--dimensional vector space of polynomials of degree $<D$. 
For $w\in\C^\times$ and $1\leqslant j\leqslant m$, we have
\begin{equation}\label{Equation:derivees}
 \left( \left( z \frac \rmd{\rmd z}\right)^tP\right)(w)=0
 \quad (0\leqslant t<t_j)
 \;
\Longleftrightarrow \;
 \left( \frac \rmd{\rmd z}\right)^tP (w)=0
 \quad (0\leqslant t<t_j).
\end{equation}
Therefore, when $w_1,\dots,w_m$ are nonzero distinct complex numbers, the linear map
\[ 
\begin{matrix}
\C[z]_{<D} & \longrightarrow &\C^D
\\
P(z)&\longmapsto&\displaystyle \left( \left( \left( z \frac \rmd{\rmd z}\right)^tP\right)(w_j)\right)_{ 0\leqslant t<t_j,\; 1\leqslant j\leqslant m}
\end{matrix}
\] 
is injective, hence its determinant in the canonical bases of $\C[z]_{<D}$ and $\C^D$, which is $M$, is not zero.

We deduce that the determinant of $M$, which is a polynomial $\Delta$ in $w_1,\dots,w_m$, vanishes only when either some $w_j$ vanishes or two $w_j$ coincide. Therefore this determinant is a constant times a product of powers of $w_j$ and a product of powers of $w_i-w_j$.

We expand this determinant as a sum of products using Laplace rule. Each $w_j$ occurs $t_j$ times with distinct exponents $\geqslant 0$, hence $\Delta$ is divisible by $w_j^{t_i(t_i-1)/2}$. Further, each monomial in this expansion has degree $D(D-1)/2$, hence this polynomial $\Delta$ is homogeneous of degree $D(D-1)/2$. Furthermore this expansion shows that the degree of $\Delta$ in $w_j$ is 
\[
(D-1)+(D-2)+\cdots+(D-t_j)=
 t_j(D-(t_j+1)/2).
\]
We claim that this polynomial is divisible by $(w_i-w_j)^{t_it_j}$ for $1\leqslant i<j\leqslant m$. To prove this claim, we consider the values at $w_j$ of the derivatives of $\Delta$ with respect to $w_i$. Using \eqref{Equation:derivees}, we replace $\partial/\partial w_i$ with $w_i \partial/\partial w_i$. The column of index $(t,i)$ with $0\leqslant t<t_i$ is the transpose of 
\begin{equation}\label{Equation:colonne(t,i)}
\bigl(\delta_{t,0},w_i, 2^t w_i^2,3^tw_i^3,\dots, (D-1)^tw_i^{D-1}\bigr)
\end{equation}
For $0\leqslant t\leqslant t_i-1$, let $\tau_t\geqslant 0$. If we apply the operator $(w_i \partial/\partial w_i)^{\tau_t}$ to the column which is the transpose of \eqref{Equation:colonne(t,i)} and then we substitute $w_j$ to $w_i$, we obtain the transpose of 
\begin{equation}\label{Equation:colonne(t+tau)}
\bigl(\delta_{t+\tau_t,0},w_j, 2^{t+\tau_t} w_j^2,3^{t+\tau_t} w_j^3,\dots, (D-1)^{t+\tau_t} w_j^{D-1}\bigr).
\end{equation}
Let $\tau\geqslant 0$. If the value of $(w_i \partial/\partial w_i)^\tau \Delta$ when we substitute $w_j$ to $w_i$ is not $0$, then there exists non negative integers $\tau_0,\dots,\tau_{t_i-1}$ with sum $\tau$ such that the $t_i-1$ columns \eqref{Equation:colonne(t+tau)} are pairwise distinct and are not equal to one the columns of index $(t',j)$ as long as $0\leqslant t'< t_j$. Therefore we have $\tau_t+t\geqslant t_j$ for $0\leqslant t<t_i$ and $\tau_0, \tau_1+1,\dots,\tau_{t_i}-1+t_i-1$ are pairwise distinct. This implies $\tau_0+\tau_1+\cdots+\tau_{t_i-1}\geqslant t_it_j$, and our claim is proved. 

We recover the partial degree of $\Delta$ with respect to $w_i$:
\[
t_i\sum_{j\not=i} t_j +\frac {t_i(t_i-1)} 2=
t_iD-\frac {t_i(t_i+1)} 2\cdotp
\]
From
\[
D^2=\left(\sum_{j=1}^m t_j\right)^2=\sum_{j=1}^m t_j^2+2\sum_{1\leqslant i < j\leqslant m} t_it_j
\]
we deduce
\[
\frac{D(D-1)}2 =\sum_{j=1}^m \frac{ t_j(t_j-1)} 2+\sum_{1\leqslant i < j\leqslant m} t_it_j
\]
and \eqref{Equation:determinant} follows.

For $t\geqslant 0$, the Vandermonde determinant of $0,1,2,\dots,t-1$, namely
 \[
 k(t):=\det
 \begin{pmatrix}
1&0&0&\cdots&0
\\
1&1&1&\cdots&1
\\ 
 1&2&2^2&\cdots&2^{t-1}
\\
\vdots&\vdots&\vdots&\ddots&\vdots
\\
1&(t-1) &(t-1)^2 &\cdots&(t-1)^{t-1} 
\end{pmatrix}
\]
has the value \footnote{This sequence $ \bigl(k(t)\bigr)_{t\geqslant 0}$ is the sequence \url{https://oeis.org/A000178} of the on-line Encyclopedia of Integer Sequences.}
 \[
 k(t)=1^{t-1}2^{t-2}\cdots (t-2)^2(t-1)=\prod_{i=1}^{t-1}i^{t-i}
 =\prod_{j=1}^{t-1} j!.
 \]
By induction we deduce
 \[
 k(t_1,\cdots,t_m)=\prod_{j=1}^mk(t_j).
 \]
 \hfill $\square$
 \end{proof}

From Lemma \ref{Equation:determinant} we deduce:

\begin{lemma}\label{Lemma:vanderMonde1variable1}
Let $w_1,\dots,w_m$ be nonzero complex numbers, $T$ a positive integer and $q_{t,j}$ ($0\leqslant t\leqslant T$, $1\leqslant j\leqslant m$) complex numbers, not all of which are $0$. There exists a constant $\eta>0$ with the following property. Let $A$ be a real number. For $0\leqslant a\leqslant m(T+1)$, define 
\[
\xi_a=\sum_{t=0}^T \sum_{j=1}^m q_{t,j}(A+a)^tw_j^a.
\]
{\rm (i)} 
Assume
\[
\xi_a=0\quad \hbox{for}\quad 0\leqslant a\leqslant m(T+1).
\]
Then $w_1,\dots,w_m$ are not pairwise distinct.
\\
{\rm (ii)} 
Assume that the numbers $\xi_a$ ($0\leqslant a\leqslant m(T+1)$) are not all $0$. Then 
\[
\max_{0\leqslant a\leqslant m(T+1)}
|\xi_a |\geqslant \eta \max\{1,A\}^{-mT(T+1)}.
\]
\end{lemma}

\begin{proof} 
We use Lemma \ref{Lemma:Vandermonde} with $t_1=\dots=t_m=L+1$ for the sequence of polynomials $r_t(T)=(A+T)^t$ $(t\geqslant 0)$. 
\\
{\rm (i)} 
If there is a nontrivial solution $q_{t,j}$ to the linear system of equations $\xi_a=0$, the determinant of this system is zero. Lemma \ref{Equation:determinant} implies that $w_1,\dots,w_m$ are not pairwise distinct.
\\
{\rm (ii)} 
Without loss of generality, reordering the sequence $w_1,\dots,w_m$ if necessary, we may assume 
 \[
 \{w_1,\dots,w_m\}=\{w_1,\dots,w_\mu\}
 \]
 with $1\leqslant \mu\leqslant m$. We have
 \[
\xi_a=\sum_{t=0}^T \sum_{j=1}^\mu \widetilde{q}_{t,j}(A+a)^tw_j^a
\]
with 
\[
\widetilde{q}_{t,j}=\sum_{\atop{1\leqslant i\leqslant m}{w_i=w_j}} q_{t,i}
\quad\text{for}\quad 0\leqslant t\leqslant T,\; 1\leqslant j\leqslant \mu.
\]
Since the numbers $\xi_a$ ($0\leqslant a\leqslant m(T+1)$) not all vanish, the coefficients $\widetilde{q}_{t,j}$ are not all zero. The nonzero values $|\widetilde{q}_{t,j}|$ are bounded from below by 
\[
\min_{I}\left|\sum_{i\in I}q_{t,i}\right|,
\]
where the minimum is over the subsets $I$ of $ \{1,\dots,m\}$ such that $\sum_{i\in I}q_{t,i}\not=0$; this minimum is a positive constant which does not depend on $A$. 
Since $w_1,\dots,w_\mu$ are pairwise distinct, the determinant of the system of linear forms $\xi_a$ is not zero, and has an absolute value bounded from below by a positive constant, as shown by Lemma \ref{Lemma:Vandermonde}. Inverting the matrix shows that there is a constant $c>0$ such that
\[
c\leqslant \max\{1,A\}^{mT(T+1)} \max_{0\leqslant a\leqslant m(T+1)}
|\xi_a |.
\] 
 \hfill $\square$
 \end{proof}

\begin{lemma}\label{Lemma:vanderMonde1variable2}
Let $v_1,\dots,v_n$ be nonzero complex numbers. Let $Q\in\C[X_0,X_1,\dots,X_n]$ be a nonzero polynomial with partial degrees $\leqslant L$. There exist two positive constants $c_1$ and $c_2$ depending on $v_1,\dots,v_n$ and $Q$ such that the following holds. Let $A$ be a positive real number. Assume
\[ 
|Q(A+a, v_1^{A+a},\dots,v_n^{A+a})|<c_1 c_2^{-A}
\quad\hbox{for any $a\in\Z$ satisfying $0\leqslant a<(L+1)^{n+1}$}.
\] 
Then 
\[
Q(A+a, v_1^{A+a},\dots,v_n^{A+a})=0 
\quad\hbox{for any $a\in\Z$ satisfying $0\leqslant a<(L+1)^{n+1}$}
\]
and there exists $(h_1,\dots,h_n)\in\Z^n$ satisfying
\[
0<\max_{1\leqslant i\leqslant n} |h_i|\leqslant L
\quad\hbox{
and
}\quad
v_1^{h_1}\cdots v_n^{h_n}=1.
\]
\end{lemma}

\begin{proof} 
Let us write
\[
Q(X_0,X_1,\dots,X_n)=\sum_{\uell}q_{\uell}
X_0^{\ell_0}X_1^{\ell_1}\cdots X_n^{\ell_n},
\]
where $\uell$ stands for $(\ell_0,\ell_1,\dots,\ell_n)$, $0\leqslant \ell_i\leqslant L$ ($0\leqslant i\leqslant n$). 
Then
\[
Q(A+a, v_1^{A+a},\dots,v_n^{A+a})=
\sum_{\uell}\widetilde{q}_{\uell}
(A+a)^{\ell_0}(v_1^{\ell_1}\cdots v_n^{\ell_n})^a,
\]
where 
\[
\widetilde{q}_{\uell}=q_{\uell}(v_1^{\ell_1}\cdots v_n^{\ell_n})^A. 
\]
We use Lemma \ref{Lemma:vanderMonde1variable1} with 
 $m=(L+1)^n$ and the following numbers $w_1,\dots,w_m$: 
\[
v_1^{\ell_1}\cdots v_n^{\ell_n}, \qquad 0\leqslant \ell_i\leqslant L, \quad 1\leqslant i\leqslant n.
\] 
 \hfill $\square$
 \end{proof}

Notice that conversely, if there exists $(h_1,\dots,h_n)\in\Z^n\smallsetminus\{0\}$ satisfying
 \[
v_1^{h_1}\cdots v_n^{h_n}=1,
\]
then 
\[
\prod_{h_i>0} X_i^{h_i}-\prod_{h_i<0} X_i^{|h_i|}
\]
is a nonzero polynomial $Q$ of partial degree $|h_i|$ in $X_i$ and $0$ in $X_0$ which satisfies, for any $a\in\Z$, 
\[
Q(a,v_1^a,\dots,v_n^a)=0.
\]

We need a version of Lemma \ref{Lemma:vanderMonde1variable2} for two variables. Notice that elements $(v_1,w_1), \dots, (v_n,w_n)$ in $(\C^\times)^2$ are multiplicatively dependent if and only if there exists $(h_1,\dots,h_n)\in\Z^n\smallsetminus\{0\}$ such that 
\[
v_1^{h_1}\cdots v_n^{h_n}= w_1^{h_1}\cdots w_n^{h_n}=1,
\]
which is equivalent to 
\[
(v_1^aw_1^b)^{h_1}\cdots (v_n^aw_n^b)^{h_n}=1
\]
for all $(a,b)\in\Z^2$.

\begin{lemma}\label{Lemma:vanderMonde2variables}
	Let $(v_1,w_1),\dots,(v_n,w_n)$ be $n$ elements in $(\C^\times)^2$. The following assertions are equivalent.
	\\
	{\rm (i)} There exists a nonzero polynomial $P\in\C[X_1,\dots,X_n]$ such that, for any $(a,b)\in\Z^2$, 
\[ 
	P(v_1^aw_1^b,\dots,v_n^aw_n^b)=0.
\] 	 
	{\rm (ii)} There exists a nonzero polynomial $Q\in\C[T_1,T_2,X_1,\dots,X_n]$ such that, for any $(a,b)\in\Z^2$, 
\[ 
	Q(a,b, v_1^aw_1^b,\dots,v_n^aw_n^b)=0.
\] 
	{\rm (iii)}
	The elements $(v_1,w_1), \dots, (v_n,w_n)$ in $(\C^\times)^2$ are multiplicatively dependent.
\end{lemma}

\begin{proof}
 {\rm (i)} $\Rightarrow $ {\rm (ii)} is plain. 
 
 \noindent 
	 {\rm (iii)} $\Rightarrow $ {\rm (i)}: if there exists $(h_1,\dots,h_n)\in\Z^n\smallsetminus\{0\}$ such that 
\[ 
	(v_1^aw_1^b)^{h_1}\cdots (v_n^aw_n^b)^{h_n}=1
\] 
	for all $(a,b)$ in $\Z^2$, then the polynomial 
\[ 
	P(X_1,\dots,X_n)=
	\prod_{h_i>0} 
	X_i^{h_i}-
	\prod_{h_i<0} 
	X_i^{-h_i}
\] 
	satisfies property {\rm (i)}. 
	
	\noindent {\rm (ii)} $\Rightarrow $ {\rm (iii)}: assume $Q$ satisfies property {\rm (ii)}. Let $L$ be an upper bound for the partial degrees of $Q$.
	For each $ a \in\Z$, we have, for all $b\in\Z$, 
\[ 
	Q(ab, b, (v_1^aw_1)^b,\dots,(v_n^aw_n)^b)=0.
\] 
	From Lemma \ref{Lemma:vanderMonde1variable2} with $v_i$ replaced with $v_i^aw_i$, it follows that 
 there exists $(\ell_1^{(a)},\dots,\ell_n^{(a)})\in\Z^n\smallsetminus\{0\}$ such that 
\[ 
	(v_1^aw_1)^{\ell_1^{(a)}}\cdots (v_n^aw_n)^{\ell_n^{(a)}}=1
\quad\hbox{
	and 
}\quad
	\max_{1\leqslant i\leqslant n} |\ell_i^{(a)}|\leqslant L.
\] 
	There are at most $(L+1)^n$ tuples $(\ell_1^{(a)},\dots,\ell_n^{(a)})$. 
From Dirichlet's box principle we deduce that there exist two distinct integers $a_1$ and $a_2$ in $\Z$ with $0\leqslant a_1<a_2\leqslant (L+1)^n$ for which $\ell_i^{(a_1)}=\ell_i^{(a_2)}$ for $1\leqslant i\leqslant n$:
\[ 
	(v_1^{a_1}w_1 )^{ \ell_1^{(a_1)}} \cdots (v_n^{a_1}w_n )^{ \ell_n^{(a_1)}}=
	(v_1^{a_2}w_1 )^{ \ell_1^{(a_1)}} \cdots (v_n^{a_2}w_n )^{ \ell_n^{(a_1)}}
	=1.
\] 
The subgroup $\Z(a_1,1)+\Z(a_2,1)$ has index $a_2-a_1$ in $\Z^2$.
For $i=1,\dots,n$ set $h_i=(a_2-a_1)\ell_i^{(a_1)}$. For $(a,b)\in\Z^2$ we have 
\[
(a_2-a_1)(a,b)=(-a+ba_2)(a_1,1)+(a-ba_1)(a_2,1),
\]
hence
\[
\prod_{i=1}^n (v_i^aw_i^b )^{h_i}= \prod_{i=1}^n (v_i^aw_i^b )^{(a_2-a_1)\ell_i^{(a_1)}}
=
\prod_{i=1}^n (v_i^{a_1}w_i )^{ \ell_i^{(a_1)}(-a+ba_2)}
\prod_{i=1}^n (v_i^{a_2}w_i )^{ \ell_i^{(a_1)}(a-ba_1)}=1.
\]
 \hfill $\square$
 \end{proof}

\subsection{Without the sigma function}

We give a necessary and sufficient condition for functions $z,	\wp(z), \zeta(z),
	\rme^{t_1z},\dots,\rme^{t_rz},f_{u_1}(z),\dots,f_{u_s}(z)$ to be algebraically independent.

\begin{proposition}\label{Proposition:independance_f_u}
	Let $t_1,\dots,t_r$, $u_1,\dots,u_s$ be complex numbers with $u_i\not\in\Omega$ for $1\leqslant i\leqslant s$. 
	The following conditions are equivalent.\\
	{\rm (i)} The $3+r+s $ functions 
\[ 
z,	\wp(z), \zeta(z),
	\rme^{t_1z},\dots,\rme^{t_rz},f_{u_1}(z),\dots,f_{u_s}(z)
\] 
	are algebraically dependent.
	\\
	{\rm (ii)} 
	There exists $(h_1,\dots,h_{r+s})\in\Z^{r+s}\smallsetminus\{0\}$ such that the meromorphic function
\[ 
	\rme^{(h_1t_1+\cdots+h_rt_r)z}
	f_{u_1}(z)^{h_{r+1}}\cdots f_{u_s}(z)^{h_{r+s}}
\] 
	belongs to $\C\bigl(\wp(z),\wp'(z)\bigr)$.
	\\
	{\rm (iii)} The $r+s$ elements 
\[ 
	\exp(t_1\bfomega), \dots, \exp(t_r\bfomega), \; 
	\exp(\bflambda(u_1,\bfomega) ),\ldots, \exp(\bflambda(u_s,\bfomega) )
\] 
	of $(\C^\times)^\Omega$ are multiplicatively dependent. 
	\\
	{\rm (iv)} The $2+r+s$ elements 
\[
\begin{pmatrix} 2\pi\rmi \\ 0 \end{pmatrix}, 
\begin{pmatrix} 0\\ 2\pi\rmi \end{pmatrix}, 
\begin{pmatrix} t_1\omega_1\\ t_1\omega_2\end{pmatrix}, \dots, 
\begin{pmatrix} t_r\omega_1\\ t_r\omega_2 \end{pmatrix}, \; 
\begin{pmatrix} \lambda(u_1,\omega_1)\\ \lambda(u_1,\omega_2)\end{pmatrix} ,\ldots, 
\begin{pmatrix} \lambda(u_s,\omega_1)\\ \lambda(u_s,\omega_2) \end{pmatrix} 
\]
of $\C^2$ are linearly dependent over $\Q$. 
\end{proposition}

\begin{proof}
	The equivalence {\rm (iii)} $\Leftrightarrow$ {\rm (iv)} and the implication {\rm (ii)} $\Rightarrow $ {\rm (i)} are plain.
	
	\noindent
	 {\rm (iii)} $\Rightarrow $ {\rm (ii)}:
	Assume that the $r+s$ elements 
\[ 
	\bigl(\exp(t_1\omega)\bigr)_{\omega\in \Omega}, \dots, \exp\bigl((t_r\omega)\bigr)_{\omega\in\Omega}, \; 
	\exp\bigl(\lambda(u_1,\omega)\bigr)_{\omega\in\Omega},\ldots, \exp\bigl(\lambda(u_s,\omega)\bigr)_{\omega\in\Omega}
\] 
	of $(\C^\times)^\Omega$ are multiplicatively dependent: there exists $(h_1,\dots,h_{r+s})\in\Z^{r+s}\smallsetminus\{0\}$ such that 
\[ 
	\rme^{(h_1t_1+\cdots+h_rt_r)\omega+
		h_{r+1}\lambda(u_1,\omega)+\cdots+h_{r+s} \lambda(u_s,\omega)}=1
\] 
	for all $\omega\in\Omega$. From \eqref{Equation:lambda} it follows that the function 
\[ 
	F(z)=\rme^{(h_1t_1+\cdots+h_rt_r)z}
	f_{u_1}(z)^{h_{r+1}}\cdots f_{u_s}(z)^{h_{r+s}}
\] 
	satisfies $F(z+\omega)=F(z)$ for all $\omega\in \Omega$, and therefore belongs to $\C\bigl(\wp(z),\wp'(z)\bigr)$.

	\noindent 	
	\noindent {\rm (i)} $\Rightarrow $ {\rm (iii)}:
	 
	Assume that 
\[ 
	P(Y_1,Y_2,Y_3,X_1,\dots,X_{r+s})=\sum_{\um}\sum_{\uell} p_{\um,\uell} Y_1^{m_1} Y_2^{m_2} Y_3^{m_3}X_1^{\ell_1}\cdots X_{r+s}^{\ell_{r+s}}
\] 
	is a nonzero polynomial in $\C[Y_1,Y_2,Y_3,,X_1,\dots,X_{r+s}]$ such that the meromorphic function 
\begin{equation}\label{Equation:F(z)}
	F(z)= P\bigl(z,\wp(z),\zeta(z), \rme^{t_1z},\dots, \rme^{t_rz}, f_{u_1}(z),\dots,f_{u_s}(z)\bigr)
\end{equation}
	is $0$. For $z_0\in\C\smallsetminus\Omega$ let $Q_{z_0}(T_1,T_2,X_1,\dots,X_{r+s}) \in\C[T_1,T_2,X_1,\dots,X_{r+s}]$
	denote the polynomial 
\[ 
	\begin{aligned}
		Q_{z_0}(T_1,T_2,X_1,\dots,&X_{r+s}) = 
		\\
		&
		\sum_{\um} \sum_{\uell} p_{\um,\uell} (z_0+T_1\omega_1+T_2\omega_2)^{m_1}
		\wp(z_0)^{m_2}
		(\zeta(z_0)+T_1\eta_1+T_2\eta_2)^{m_3}
		\\
		&
		\rme^{(\ell_1t_1+\cdots+\ell_rt_r)z_0}
		\bigl(f_{u_1}(z_0)\bigr)^{\ell_{r+1}}\cdots \big(f_{u_s}(z_0)\bigr)^{\ell_{r+s}}
		X_1^{\ell_1}\cdots X_{r+s}^{\ell_{r+s}}.
	\end{aligned}
\] 
Recall Legendre relation \eqref{Equation:Legendre}: $\omega_2\eta_1-\omega_1\eta_2=2\rmi \pi$. Since the coefficients $p_{\um,\uell} $ are not all zero, there exists $z_0\in\C\smallsetminus\Omega$ such that $Q_{z_0}$ is not the zero polynomial.
For $\omega=a\omega_1+b\omega_2\in\Omega$ we deduce from $F(z_0+\omega)=0$:
\[ 
	\begin{aligned}
	\sum_{\um} \sum_{\uell} p_{\um,\uell} 
	&(z_0+a\omega_1+b\omega_2)^{m_1}
		\wp(z_0)^{m_2}
		(\zeta(z_0)+a\eta_1+b\eta_2)^{m_3}
		\\
		&
\rme^{(\ell_1t_1+\cdots+\ell_rt_r)(z_0+\omega)}
	\bigl(f_{u_1}(z_0+\omega)\bigr)^{\ell_{r+1}}\cdots \bigl(f_{u_s}(z_0+\omega)\bigr)^{\ell_{r+s}}=0
\end{aligned}
\] 
 which, by \eqref{Equation:lambda}, is
\[ 
	Q_{z_0}(a,b,\rme^{t_1\omega},\dots,\rme^{t_r\omega}, 
	\rme^{\lambda(u_1,\omega)},\dots,\rme^{\lambda(u_s,\omega)})=0.
\] 
The conclusion follows from Lemma \ref{Lemma:vanderMonde2variables}.
 \hfill $\square$
 \end{proof}

\begin{remark}\label{Remark:majorationcoefficients}
The proof of {\rm (i)} $\Rightarrow $ {\rm (iii)} together with Lemma \ref{Lemma:vanderMonde1variable2} shows that it the equation \eqref{Equation:F(z)} is satisfied, then there exists $(h_1,\dots,h_{r+s})\in\Z^{r+s}\smallsetminus\{0\}$ such that 
\[
(h_1t_1+\cdots+h_rt_r)\omega+h_{r+1}\lambda(u_1,\omega)+\cdots+h_{r+s}\lambda(u_s,\omega)\in 2\pi\rmi\Z
\]
for all $\omega\in\Omega$ and 
\[
\max\{|h_1|,\dots,|h_{r+s}|\}\leqslant L
\]
where $L$ is an upper bound for the partial degrees of $P$. 
\end{remark}

\begin{corollary} 
Let $t_1,\dots,t_r$, $u_1,\dots,u_s$ be complex numbers. Assume $u_j\not\in\Omega$ for $1\leqslant j\leqslant s$. 
Let $i_1,\dots,i_\varrho$, $j_1,\dots,j_\sigma$ be indices such that 
\[
\begin{pmatrix} 2\pi\rmi \\ 0 \end{pmatrix}, 
\begin{pmatrix} 0\\ 2\pi\rmi \end{pmatrix}, 
\begin{pmatrix} t_{i_1}\omega_1\\ t_{i_1}\omega_2\end{pmatrix}, \dots, 
\begin{pmatrix} t_{i_\varrho}\omega_1\\ t_{i_\varrho}\omega_2 \end{pmatrix}, \; 
\begin{pmatrix} \lambda(u_{j_1},\omega_1)\\ \lambda(u_{j_1},\omega_2)\end{pmatrix} ,\ldots, 
\begin{pmatrix} \lambda(u_{j_\sigma},\omega_1)\\ \lambda(u_{j_\sigma},\omega_2) \end{pmatrix} 
\]
is a basis of the $\Q$--vector space spanned by
\[
\begin{pmatrix} 2\pi\rmi \\ 0 \end{pmatrix}, 
\begin{pmatrix} 0\\ 2\pi\rmi \end{pmatrix}, 
\begin{pmatrix} t_1\omega_1\\ t_1\omega_2\end{pmatrix}, \dots, 
\begin{pmatrix} t_r\omega_1\\ t_r\omega_2 \end{pmatrix}, \; 
\begin{pmatrix} \lambda(u_1,\omega_1)\\ \lambda(u_1,\omega_2)\end{pmatrix} ,\ldots, 
\begin{pmatrix} \lambda(u_s,\omega_1)\\ \lambda(u_s,\omega_2) \end{pmatrix}. 
\]
Then 
\[
\bigl\{z, \wp(z), \zeta(z), \rme^{t_{i_1}z},\dots, \rme^{t_{i_\varrho}z}, f_{u_{j_1}}(z),\dots,f_{u_{j_\sigma}}(z)\bigr\}
\]
is a transcendence basis over $\C$ of the field 
\[
\C\bigl(z, \wp(z), \zeta(z), \rme^{t_1z},\dots, \rme^{t_rz}, f_{u_1}(z),\dots,f_{u_s}(z)\bigr).
\]
\end{corollary}

\begin{proof}
By renumbering the sequences $t_1,\dots,t_r$ and $u_1,\dots,u_s$ if necessary, there is no loss of generality to assume that 
$(i_1,\dots,i_\varrho)=(1,\dots,\varrho)$ and $(j_1,\dots,j_\sigma)=(1,\dots,\sigma)$. By the implication (iv)$\Rightarrow$ (i) of Proposition \ref{Proposition:independance_f_u}, the $3+\varrho+\sigma$ functions 
\begin{equation}\label{Equation:basetranscendance}
z, \wp(z), \zeta(z), \rme^{t_1z},\dots, \rme^{t_\varrho z}, f_{u_1}(z),\dots,f_{u_\sigma}(z)
\end{equation}
are algebraically independent. It only remains to prove that for $\varrho<i\leqslant r$ in case $r>\varrho$ (resp. $\sigma<j\leqslant s$ in case $s>\sigma$) the function $\rme^{t_iz}$ (resp. $f_{u_j}(z)$) is algebraic over the field generated over $\C$ by the functions \eqref{Equation:basetranscendance}. 

Assume $r>\varrho$ and let $i$ satisfy $\varrho<i\leqslant r$. By assumption the $3+\varrho+\sigma$ elements
\[
\begin{pmatrix} 2\pi\rmi \\ 0 \end{pmatrix}, 
\begin{pmatrix} 0\\ 2\pi\rmi \end{pmatrix}, 
\begin{pmatrix} t_1\omega_1\\ t_1\omega_2\end{pmatrix}, \dots, 
\begin{pmatrix} t_\varrho\omega_1\\ t_\varrho\omega_2 \end{pmatrix}, \; 
\begin{pmatrix} t_i\omega_1\\ t_i\omega_2 \end{pmatrix}, \; 
\begin{pmatrix} \lambda(u_1,\omega_1)\\ \lambda(u_1,\omega_2)\end{pmatrix} ,\ldots, 
\begin{pmatrix} \lambda(u_\sigma,\omega_1)\\ \lambda(u_\sigma,\omega_2) \end{pmatrix}
\]
are linearly dependent over $\Q$. From (v)$\Rightarrow$(1) in
Proposition \ref{Proposition:independance_f_u}, it follows that the $4+\varrho+\sigma$ functions 
\[
z, \wp(z), \zeta(z), \rme^{t_1z},\dots, \rme^{t_\varrho z}, \rme^{t_i z}, f_{u_1}(z),\dots,f_{u_\sigma}(z)
\]
are algebraically dependent. Since the functions \eqref{Equation:basetranscendance} are algebraically independent, it follows that the function $\rme^{t_i z}$ is algebraic over the field generated by the functions \eqref{Equation:basetranscendance}. 

In the same way, if $s>\sigma$, for $\sigma<j\leqslant s$, the $3+\varrho+\sigma$ elements
\[
\begin{pmatrix} 2\pi\rmi \\ 0 \end{pmatrix}, 
\begin{pmatrix} 0\\ 2\pi\rmi \end{pmatrix}, 
\begin{pmatrix} t_1\omega_1\\ t_1\omega_2\end{pmatrix}, \dots, 
\begin{pmatrix} t_\varrho\omega_1\\ t_\varrho\omega_2 \end{pmatrix}, \; 
\begin{pmatrix} \lambda(u_1,\omega_1)\\ \lambda(u_1,\omega_2)\end{pmatrix} ,\ldots, 
\begin{pmatrix} \lambda(u_\sigma,\omega_1)\\ \lambda(u_\sigma,\omega_2) \end{pmatrix}, \;
\begin{pmatrix} \lambda(u_j,\omega_1)\\ \lambda(u_j,\omega_2) \end{pmatrix}
\]
are linearly dependent over $\Q$. 
Proposition \ref{Proposition:independance_f_u} implies that the $4+\varrho+\sigma$ functions 
\[
z, \wp(z), \zeta(z), \rme^{t_1z},\dots, \rme^{t_\varrho z}, f_{u_1}(z),\dots,f_{u_\sigma}(z), f_{u_j}(z),
\]
are algebraically dependent. Since the functions \eqref{Equation:basetranscendance} are algebraically independent, it follows that the function $f_{u_j}(z)$ is algebraic over the field generated by the functions \eqref{Equation:basetranscendance}.

 \hfill $\square$
 \end{proof}

We show that the assumption in Theorem \ref{Theorem:AlgebraicIndependenceFunctions} that $\omega_1,\omega_2,u_1,\dots,u_s$ are linearly independent over $\Q$ cannot be omitted (see the discussion in \cite{M} after Lemma 20.10 p.~257).

 \begin{proposition}\label{Proposition:AlgebraicIndependenceSerreFunctions}
	Let $u_1,\dots,u_s$ be complex numbers which
do not belong to $ \Omega$ but such that $\omega_1,\omega_2$, $u_1,\dots,u_s$ are linearly dependent over $\Q$. Then there exist $(k_1,\dots,k_r)\in\Z^r\smallsetminus\{0\}$ and $t_0\in\C$ such that
	the function
\[ 
	\rme^{t_0z} f_{u_1}^{k_1} (z)\dots f_{u_s}^{k_s}(z)
\] 
	belongs to $\C\bigl(\wp(z),\wp'(z)\bigr)$.
	
	\end{proposition}

\begin{proof}
	Assume that $\omega_1,\omega_2,u_1,\dots,u_s$ are linearly dependent over $\Q$:
\[ 
	k_1 u_1+\cdots+k_su_s=\omega_0\in\Omega
\] 
with $(k_1,\dots,k_r)\in\Z^r\smallsetminus\{0\}$. 
	Let 
\[ 
	t_0=k_1 \zeta(u_1)+\cdots+k_s\zeta(u_s)- \eta(\omega_0).
\] 
	Then for $\omega\in\Omega$, 
\[ 
	t_0\omega+k_1\lambda(u_1,\omega)+\cdots+k_s\lambda(u_s,\omega)=\eta(\omega)\omega_0-\eta(\omega_0)\omega\in 2\pi \rmi\Z.
\] 
We deduce that the function 
\[ 
	\rme^{t_0z} f_{u_1}^{k_1} (z)\dots f_{u_s}^{k_s}(z)
\] 
is periodic with group of periods containing $\Omega$, hence belongs to $\C\bigl(\wp(z),\wp'(z)\bigr)$. 
 \hfill $\square$
 \end{proof}

\begin{example}
Algebraic dependence of the three function $\rme^{tz}$, $\wp(z)$ and $f_u(z)$ when $u$ is a torsion point.
\\ 
\rm
Let $u\in(\Omega \otimes_\Z\Q)\smallsetminus\Omega$. There exists an integer $h\geqslant 2$, such that $\omega_0:=hu$ belongs to $\Omega$. Set
$t_0=h\zeta(\omega_0/h)-\eta(\omega_0) $. Proposition \ref{Proposition:AlgebraicIndependenceSerreFunctions} shows that the function 
$f_u(z)^h\rme^{t_0 z}$ belongs to the field $\C\bigl(\wp(z),\wp'(z)\bigr)$. 

Using \cite{Wald79b} (see \cite{BW2} for more details) one checks that $t_0$ belongs to $\Q\bigl(g_2,g_3,\wp(u),\wp'(u)\bigr)$ and that the function 
$f_u^{h^2}(z)\rme^{ht_0 z}$ belongs to the field $\Q\bigl(g_2,g_3,\wp(u),\wp'(u), \wp(z),\wp'(z)\bigr)$. 

From the algebraic independence of the two functions $\rme^{t_1z}$ and $\rme^{t_2z}$, when $t_1$ and $t_2$ are linearly independent over $\Q$, it follows that the set of $t\in\C^\times$ for which the three functions $\rme^{tz}, \wp(z), f_u(z)$ are algebraically dependent is the set $\Q^\times t_0$ of nonzero rational multiples of $t_0$. 
\end{example}

We now prove part 1 of Theorem \ref{Theorem:AlgebraicIndependenceFunctions} without the sigma function.

\begin{proposition}\label{Proposition:sansCMnisigma}
Let $t_1,\dots,t_r$ be complex numbers linearly independent over $\Q$. Let $u_1,\dots,u_s$ be complex numbers such that $\omega_1,\omega_2,u_1,\dots,u_s$ are linearly independent over $\Q$.
	Then the $3+r+s$ functions 
\[ 
	z, \wp(z), \zeta(z), \rme^{t_1z},\dots, \rme^{t_rz}, f_{u_1}(z),\dots,f_{u_s}(z)
\]
	are algebraically independent. 
\end{proposition}

\begin{proof}
 Assume that 
 $t_1,\dots,t_r$ are linearly independent over $\Q$ and that the $3+r+s$ functions 
\[ 	z, \wp(z), \zeta(z), \rme^{t_1z},\dots, \rme^{t_rz}, f_{u_1}(z),\dots,f_{u_s}(z)
\] 
	are algebraically dependent. 
From Proposition \ref{Proposition:independance_f_u} we deduce that the $2+r+s$ elements 
\[ 
	\begin{pmatrix} 2\pi\rmi \\ 0 \end{pmatrix}, 
	\begin{pmatrix} 0\\ 2\pi\rmi \end{pmatrix}, 
	\begin{pmatrix} t_1\omega_1\\ t_1\omega_2\end{pmatrix}, \dots, 
	\begin{pmatrix} t_r\omega_1\\ t_r\omega_2 \end{pmatrix}, \; 
	\begin{pmatrix} \lambda(u_1,\omega_1)\\ \lambda(u_1,\omega_2)\end{pmatrix} ,\ldots, 
	\begin{pmatrix} \lambda(u_s,\omega_1)\\ \lambda(u_s,\omega_2) \end{pmatrix} 
\] 
	of $\C^2$ are linearly dependent over $\Q$. Hence there exist rational integers $h_0,h'_0,h_1,\dots,h_{r+s}$, not all zero, such that
\[
	\left\{
	\begin{aligned}
		2\pi \rmi h_0+(h_1t_1+\cdots+h_rt_r)\omega_1+h_{r+1}\lambda(u_1,\omega_1)+\cdots+h_{r+s}\lambda(u_s,\omega_1)&=0,
		\\
		2\pi \rmi h'_0+(h_1t_1+\cdots+h_rt_r)\omega_2+h_{r+1}\lambda(u_1,\omega_2)+\cdots+h_{r+s}\lambda(u_s,\omega_2)&=0.
	\end{aligned}
	\right.
\] 
Since $t_1,\dots,t_r$ are linearly independent over $\Q$ and since $\omega_2/\omega_1\not\in\Q$, the numbers $h_{r+1},\dots,h_{r+s}$ are not all zero. 
	We eliminate $h_1t_1+\cdots+h_rt_r$ by multiplying the first equation by $\omega_2$ and subtracting the second equation multiplied by $\omega_1$. From \eqref{Equation:ConsequenceLegendre} we deduce
\[ 
	h_{r+1} u_1+\cdots+h_{r+s} u_s = h_0\omega_2-h'_0\omega_1,
\] 
which proves that $\omega_1,\omega_2,u_1,\dots,u_s$ are linearly dependent over $\Q$.	
\hfill$\square$
\end{proof}
 
 \bigskip
 
The following result, which is a consequence of Proposition \ref{Proposition:independance_f_u} and Remark \ref{Remark:majorationcoefficients}, shows that there is no formula expressing $f_{u_1+u_2}(z)$ in terms of $f_{u_1}(z)$ and $f_{u_2}(z)$. As a function of the variable $u$ in $\C$, $f_u(z)$ is not meromorphic in $\C$: there is an essential singularity at the points $u$ in $\Omega$. However $f_u(z)$ is an analytic function of $(u,z)$ in $(\C\smallsetminus\Omega)^2$.

\begin{corollary} Let $\wp$ be an elliptic function. 
\\
(a)
The six functions 
\[
z,\wp(z),\zeta(z), f_{u_1}(z), f_{u_2}(z), f_{u_1+u_2}(z)
\]
of three variables $(z,u_1,u_2)$ in the domain 
\[
\{(z,u_1,u_2)\in\C^3\; \mid \; z\not\in\Omega, \; u_1\not\in\Omega, \; u_2\not\in\Omega, \; u_1+u_2\not\in\Omega\}
\]
are algebraically independent over $\C$.
\\
(b)
Let $\lambda\in\C\smallsetminus\{0,1,-1\}$. Then the five functions 
\[
z,\wp(z),\zeta(z), f_{u}(z), f_{\lambda u}(z)
\]
of two variables $(z,u)$ in the domain 
\[
\{(z,u)\in\C^2\; \mid \; z\not\in\Omega, \; u\not\in\Omega, \; \lambda u\not\in\Omega\}
\]
are algebraically independent over $\C$.
\\
(c)
Let $s\geqslant 1$ and let $u_1,\dots,u_s$ be $s$ elements in $\C\smallsetminus\Omega$. The two following assertions are equivalent.
\\
{\rm (i)} The meromorphic functions 
\[
z,\wp(z),\zeta(z), f_{u_1}(z),\dots, f_{u_s}(z)
\]
of a single variable $z$ in $\C$ are algebraically dependent over $\C$.
\\
{\rm (ii)} There exists $(a_1,\dots,a_s)\in \Z^s\smallsetminus\{0\}$ such that 
\begin{equation}\label{Equation:a1u1}
a_1u_1+\cdots+a_su_s\in\Omega\quad\text{and}\quad a_1\zeta(u_1)+\cdots+a_s\zeta(u_s)=\eta(a_1u_1+\cdots+a_su_s).
\end{equation}
\\
If property {\rm (ii)} is satisfied, then the function 
\[
f_{u_1}(z)^{a_1}\cdots f_{u_s}(z)^{a_s}
 \]
belongs to $\C\bigl(\wp(z),\wp'(z)\bigr)$.
\end{corollary}
 
 \begin{proof}
 \null\hfill\break
(a) Assume that there is a nonzero polynomial $P\in\C[X_0,X_1,X_2,Y_1,Y_2,Y_3]$ satisfying
 \[
 P\bigl(z,\wp(z),\zeta(z), f_{u_1}(z), f_{u_2}(z), f_{u_1+u_2}(z)\bigr)=0.
 \]
 From Proposition \ref{Proposition:independance_f_u} and Remark \ref{Remark:majorationcoefficients}, we deduce that for all $(u_1,u_2)\in (\C\smallsetminus\Omega)^2$ such that $u_1+u_2\not\in\Omega$, there exists $(a_1,a_2,a_3)\in\Z^3\smallsetminus\{0\}$ with $\max\{|a_1|,|a_2|,|a_3|\}$ bounded above by the partial degrees of $P$ with
\begin{equation}\label{Equation:a_i-lambda}
 a_1\lambda(u_1,\omega)+a_2\lambda(u_2,\omega)+a_3\lambda(u_1+u_2,\omega)\in 2\pi\rmi\Z
 \end{equation}
 for all $\omega\in \Omega$. 
 
 Writing the relation 
 \[
\bigl(a_1\zeta(u_1)+a_2\zeta(u_2) + a_3\zeta(u_1+u_2)\bigr) \omega-
\bigl(a_1 u_1 +a_2 u_2+a_3(u_1+u_2)\bigr)\eta(\omega)\in 2\pi\rmi\Z
 \]
 for $\omega_1$ and $\omega_2$ yields the existence of two rational integers $b_1$ and $b_2$ such that 
 \[
 \begin{aligned}
 a_1 u_1 +a_2 u_2+a_3(u_1+u_2)&=b_1\omega_1+b_2\omega_2,
 \\
 a_1\zeta(u_1)+a_2\zeta(u_2) + a_3\zeta(u_1+u_2)&=b_1\eta_1+b_2\eta_2.
 \end{aligned}
 \]
 One deduces
\begin{equation}\label{Equation:zeta(z+au)}
 \zeta\bigl(z+a_1 u_1 +a_2 u_2+a_3(u_1+u_2)\bigr)=\zeta(z)+a_1\zeta(u_1)+a_2\zeta(u_2) + a_3\zeta(u_1+u_2).
\end{equation}
 We fix $z\in\C\smallsetminus\Omega$ and we let $u_2$ be sufficiently small not $0$ so that $u_2$ and $z+(a_2+a_3)u_2$ are not poles of $\wp$. We let $u_1$ tend to $0$. The existence of an upper bound for $\max\{|a_1|,|a_2|,|a_3|\}$, independent on $u_1$ and $u_2$, shows that the left hand side of \eqref{Equation:zeta(z+au)} is bounded above and that $\zeta(z)+a_2\zeta(u_2) + a_3\zeta(u_1+u_2)$ on the right hand side is also bounded above.
 Hence $a_1=0$ for $|u_1|$ and $|u_2|$ sufficiently small. In the same way $a_2=0$ for $|u_1|$ and $|u_2|$ sufficiently small. Then equation \eqref{Equation:a_i-lambda} gives a contradiction. 
 \\
 (b) Let $\lambda\in\C\smallsetminus\{0\}$. Assume that there is a nonzero polynomial $P\in\C[X_0,X_1,X_2,Y_1,Y_2]$ such that 
 \[
 P\bigl(z,\wp(z),\zeta(z), f_{u}(z), f_{\lambda u}(z)\bigr)=0.
 \]
The same argument as before, based on Proposition \ref{Proposition:independance_f_u} and Remark \ref{Remark:majorationcoefficients}, shows that for each $u\in\C\smallsetminus\Omega$ such that $\lambda u\not\in\Omega$, there exists $(a_1,a_2)\in\Z^2\smallsetminus\{0\}$, bounded by the degree of $P$, such that 
\[
 \zeta\bigl(z+(a_1 +a_2\lambda ) u \bigr)=\zeta(z)+a_1\zeta(u)+a_2\zeta(\lambda u).
 \]
 Letting $u$ tend to $0$ shows that $a_1+a_2/\lambda=0$. Fixing $u\not\in\Omega$ with $\lambda u\not\in\Omega$ and letting $z$ tend to $0$ shows that $(a_1+a_2\lambda)u\in\Omega$. This is true for all $u\not\in\Omega$ with $\lambda u\not\in\Omega$, hence $a_1+a_2\lambda=0$. We conclude $\lambda^2=1$. 
 \\
 (c) Assume \eqref{Equation:a1u1}. Then 
$a_1\lambda(u,\omega)+\cdots+a_s\lambda(u_s,\omega)\in 2\pi\rmi\Z$ for all $\omega\in\Omega$, hence the function 
\[
f_{u_1}(z)^{a_1}\cdots f_{u_s}(z)^{a_s}
 \] 
is periodic with $\Omega$ as periods, and therefore belongs to $\C\bigl(\wp(z),\wp'(z)\bigr)$. 

 Conversely, assume that there is a nonzero polynomial $P\in\C[X_0,X_1,X_2,Y_1,\dots,Y_s]$ such that 
 \[
 P\bigl(z,\wp(z),\zeta(z), f_{u_1}(z),\dots, f_{u_s}(z)\bigr)=0.
 \]
Proposition \ref{Proposition:independance_f_u} and Remark \ref{Remark:majorationcoefficients} show that there exists $(a_1,\dots,a_s)\in\Z^s\smallsetminus\{0\}$, with absolute values bounded by the degree of $P$, such that 
$a_1\lambda(u,\omega)+\cdots+a_s\lambda(u_s,\omega)\in 2\pi\rmi\Z$ for all $\omega\in\Omega$, hence
\[
 \zeta (z+a_1u +\cdots+a_su_s )=\zeta(z)+a_1\zeta(u_1)+\cdots+a_s\zeta(u_s).
 \]
 Letting $z$ tend to $0$, we deduce $a_1u +\cdots+a_su_s\in\Omega$ and \eqref{Equation:a1u1} follows. 
 \hfill $\square$
 \end{proof}
 
 \noindent
 {\em Remarks.} 
 \\
 $\bullet$ In (c), assertion (ii) with $s=1$ corresponds to $s=1$, $t_0=0$ in Proposition \ref{Proposition:AlgebraicIndependenceSerreFunctions}.
 \\
$\bullet$ Given $a_1$ and $a_2$ in $\Z\smallsetminus\{0\}$, the numbers $u\in\C\smallsetminus\Omega$ such that $v=-a_1u/a_2$ satisfies $v\not\in\Omega$ and $\zeta(v)=-a_1\zeta(u)/a_2$ are the zeroes of the derivative of the function 
 \[
 \frac{\sigma(-a_1z/a_2)^{a_2^2}}{\sigma(z)^{a_1^2}};
 \]
 see the multiplication formula for the Weierstrass zeta function, for instance \cite[Table 2]{BW1}.

\subsection{With CM and without the sigma function}
The next result implies that when the elliptic curve $E$ is CM, for $t\in\C\smallsetminus\{0\}$, $u\in\C\smallsetminus\Omega\otimes_{\Z}\Q$ and $\alpha\in\End(E)\smallsetminus\Z$, the six functions 
\[
z, \rme^{tz}, \wp(z),\zeta(z), f_u(z), f_u(\alpha z)
\]
are algebraically independent. This is related with the fact that when $u$ is not a torsion point, the nontrivial endomorphisms of $E$ do not extend to endomorphisms of the extension of $E$ by $\G_m$ associated with $u$ --- see \cite{B}.

\begin{proposition}\label{Proposition:independance_f_uCM}
Let $t_1,\dots,t_r$ be complex numbers linearly independent over $\Q$ and $u_1,\dots,u_s$ complex numbers such that $\omega_1,,u_1,\dots,u_s$ are linearly independent over $k$. 
Let $\alpha_1,\dots,\alpha_s$ be elements of $\End(E)\smallsetminus\Z$. Then the $3+r+2s $ functions 
\[ 
z,	\wp(z), \zeta(z),
	\rme^{t_1z},\dots,\rme^{t_rz},f_{u_1}(z),\dots,f_{u_s}(z),f_{u_1}(\alpha_1z),\dots,f_{u_s}(\alpha_sz)
\] 
	are algebraically independent.
\end{proposition}

\begin{proof}
Since $\End(E)\cap\Q=\Z$, the assumption of Proposition \ref{Proposition:independance_f_uCM} implies that the numbers $\alpha_i$ are irrational. 

We assume that $t_1,\dots,t_r$ are linearly independent over $\Q$ and that 
\[ 
	P(Y_1,Y_2,Y_3,X_1,\dots,X_{r+2s})=\sum_{\um}\sum_{\uell} p_{\um,\uell} Y_1^{m_1} Y_2^{m_2} Y_3^{m_3}X_1^{\ell_1}\cdots X_{r+2s}^{\ell_{r+2s}}
\] 
	is a nonzero polynomial in $\C[Y_1,Y_2,Y_3,,X_1,\dots,X_{r+2s}]$ such that the meromorphic function 
\[ 
	F(z)= P\bigl(z,\wp(z),\zeta(z), \rme^{t_1z},\dots, \rme^{t_rz}, f_{u_1}(z),\dots,f_{u_s}(z),f_{u_1}(\alpha_1z),\dots,f_{u_s}(\alpha_sz)\bigr)
\] 
is $0$. We will prove that $\omega_1,,u_1,\dots,u_s$ are linearly dependent over $k$. 
	
	For $z_0\in\C\smallsetminus\Omega$ let $Q_{z_0}(T_1,T_2,X_1,\dots,X_{r+2s}) \in\C[T_1,T_2,X_1,\dots,X_{r+2s}]$
	denote the polynomial 
\[ 
	\begin{aligned}
		Q_{z_0}(T_1,T_2,&X_1,\dots,X_{r+2s}) = 
		\\
		&
		\sum_{\um} \sum_{\uell} p_{\um,\uell} (z_0+T_1\omega_1+T_2\omega_2)^{m_1}
		\wp(z_0)^{m_2}
		(\zeta(z_0)+T_1\eta_1+T_2\eta_2)^{m_3} \rme^{(\ell_1t_1+\cdots+\ell_rt_r)z_0}
		\\
		&
		\bigl(f_{u_1}(z_0)\bigr)^{\ell_{r+1}}\cdots \bigl(f_{u_s}(z_0)\bigr)^{\ell_{r+s}} 
				\bigl(f_{u_1}(\alpha_1z_0)\bigr)^{\ell_{r+s+1}}\cdots \bigl(f_{u_s}(\alpha_sz_0)\bigr)^{\ell_{r+2s}}
		X_1^{\ell_1}\cdots X_{r+2s}^{\ell_{r+2s}}.
	\end{aligned}
\] 
Since the coefficients $p_{\um,\uell} $ are not all zero, there exists $z_0\in\C\smallsetminus\Omega$ such that $Q_{z_0}$ is not the zero polynomial.
For $\omega=a\omega_1+b\omega_2\in\Omega$ we deduce from $F(z_0+\omega)=0$:
\[ 
	\begin{aligned}
	\sum_{\um} \sum_{\uell} p_{\um,\uell} 
	&(z_0+a\omega_1+b\omega_2)^{m_1}
		\wp(z_0)^{m_2}
		(\zeta(z_0)+a\eta_1+b\eta_2)^{m_3}
\rme^{(\ell_1t_1+\cdots+\ell_rt_r)(z_0+\omega)}
		\\
		&
	\bigl(f_{u_1}(z_0+\omega)\bigr)^{\ell_{r+1}}\cdots \bigl(f_{u_s}(z_0+\omega)\bigr)^{\ell_{r+s}}
	\bigl(f_{u_1}(\alpha_1(z_0+\omega))\bigr)^{\ell_{r+s+1}}\cdots \bigl(f_{u_s}(\alpha_s(z_0+\omega))\bigr)^{\ell_{r+2s}}=0.
\end{aligned}
\] 
Since $\alpha_i\omega\in\Omega$ for $1\leqslant i\leqslant s$, we may use \eqref{Equation:lambda}; hence
\[ 
	Q_{z_0}(a,b,\rme^{t_1\omega},\dots,\rme^{t_r\omega}, 
	\rme^{\lambda(u_1,\omega)},\dots,\rme^{\lambda(u_s,\omega)},
	\rme^{\lambda(u_1,\alpha_1\omega)},\dots,\rme^{\lambda(u_s,\alpha_s\omega)})=0.
\] 
From (ii)$\Rightarrow$(iii) in Lemma \ref{Lemma:vanderMonde2variables} we deduce that the $2+r+2s$ elements 
\[
\begin{aligned}
\begin{pmatrix} 2\pi\rmi \\ 0 \end{pmatrix}, 
\begin{pmatrix} 0\\ 2\pi\rmi \end{pmatrix}, 
\begin{pmatrix} t_1\omega_1\\ t_1\omega_2\end{pmatrix}, \dots, 
\begin{pmatrix} t_r\omega_1\\ t_r\omega_2 \end{pmatrix}, \; 
&
\begin{pmatrix} \lambda(u_1,\omega_1)\\ \lambda(u_1,\omega_2)\end{pmatrix} ,\ldots, 
\begin{pmatrix} \lambda(u_s,\omega_1)\\ \lambda(u_s,\omega_2) \end{pmatrix}, \; 
\\
&
\begin{pmatrix} \lambda(u_1,\alpha_1\omega_1)\\ \lambda(u_1,\alpha_1\omega_2)\end{pmatrix} ,\ldots, 
\begin{pmatrix} \lambda(u_s,\alpha_s\omega_1)\\ \lambda(u_s,\alpha_s\omega_2) \end{pmatrix} 
\end{aligned}
\]
of $\C^2$ are linearly dependent over $\Q$. Hence there exist rational integers $h_0,h'_0,h_1,\dots,h_{r+2s}$, not all zero, such that
\begin{equation}\label{Equation2ipih0}
	\left\{
	\begin{aligned}
		2\pi \rmi h_0+(h_1t_1+\cdots+h_rt_r)\omega_1+&h_{r+1}\lambda(u_1,\omega_1)+\cdots+h_{r+s}\lambda(u_s,\omega_1)
		\\
		&+h_{r+s+1}\lambda(u_1,\alpha_1\omega_1)+\cdots+h_{r+2s}\lambda(u_s,\alpha_s\omega_1)=0,
		\\
		2\pi \rmi h'_0+(h_1t_1+\cdots+h_rt_r)\omega_2+&h_{r+1}\lambda(u_1,\omega_2)+\cdots+h_{r+s}\lambda(u_s,\omega_2)
		\\&
		+h_{r+s+1}\lambda(u_1,\alpha_1\omega_2)+\cdots+h_{r+s}\lambda(u_s,\alpha_s\omega_2)=0.
	\end{aligned}
	\right.
\end{equation}
Since $t_1,\dots,t_r$ are linearly independent over $\Q$ and since $\omega_2/\omega_1\not\in\Q$, the numbers $h_{r+1},\dots,h_{r+2s}$ are not all zero. 
	We eliminate $h_1t_1+\cdots+h_rt_r$ by multiplying the first equation of \eqref{Equation2ipih0} by $\omega_2$ and subtracting the second equation multiplied by $\omega_1$. Recall \eqref{Equation:ConsequenceLegendre}. Since $\alpha_i\not=0$, the two periods $\omega'_1:=\alpha_i\omega_1$ and $\omega'_2:=\alpha_i\omega_2$ are linearly independent, hence from \eqref{Equation:LegendreGeneralise} it follows that there exist nonzero integers $k_1,\dots,k_s$ such that 
$\alpha_i\omega_2\eta(\alpha_i\omega_1)-\alpha_i\omega_1\eta(\alpha_i\omega_2)=2\pi \rmi k_i$ for $1\leqslant i\leqslant s$;
then
\[
	\lambda(u_i,\alpha_i\omega_1)\omega_2-\lambda(u_i,\alpha_i\omega_2)\omega_1=\alpha_i^{-1}k_iu_i.
\] 
Hence
\[ 
	(h_{r+1}+\alpha_1^{-1}h_{r+s+1} k_1)u_1+\cdots+(h_{r+s} +\alpha_s^{-1} h_{r+2s}k_s) u_s = h_0\omega_2-h'_0\omega_1.
\] 
Since $\alpha_i\not\in\Q$ and $k_i\not=0$ for $1\leqslant i\leqslant s$, the coefficients $h_{r+i}+\alpha_i^{-1}h_{r+s+i} k_i$ ($1\leqslant i\leqslant s$) are not all zero. 
Hence $\omega_1, u_1,\dots,u_s$ are linearly dependent over $k$.
 \hfill $\square$
 \end{proof}

\subsection{With the sigma function}

\bigskip
We will deduce Theorem \ref{Theorem:AlgebraicIndependenceFunctions} from the following result.

\begin{proposition}\label{Proposition:sigma}
Let $t_1,\dots,t_r$, $u_1,\dots,u_s$ be complex numbers with $u_i\not\in\Omega$.
\\
1.
Then the Weierstrass sigma function is transcendental over the field 
\[
\C\bigl(z, \wp(z), \zeta(z), \rme^{t_1z},\dots, \rme^{t_rz}, f_{u_1}(z),\dots,f_{u_s}(z)\bigr).
\]
2. Assume further that the elliptic curve $E$	has complex multiplication. Let $\alpha_1,\dots,\alpha_s$ be elements in $\End(E)\smallsetminus\Z$. Then
the Weierstrass sigma function is transcendental over the field 
\[
\C\bigl(z, \wp(z), \zeta(z), \rme^{t_1z},\dots, \rme^{t_rz}, f_{u_1}(z),\dots,f_{u_s}(z),f_{u_1}(\alpha_1z),\dots,f_{u_s}(\alpha_sz)\bigr).
\]
\end{proposition}
 
We will use two preliminary results. 

\begin{lemma}\label{Lemma:ZeroesMeromorphicFunctions}
Let $F$ be a nonzero meromorphic function in $\C$, let $\omega\in\C\smallsetminus\{0\}$ and let $\calU$ be a nonempty open subset of $\C$ contained in $\{z\in\C \mid 0<\Real(z/\omega)<1\}$. Then there exists $z_0\in\calU$ such that for all $a\in\Z$, $z_0+a\omega$ is neither a pole nor a zero of $F$.
\end{lemma}

\begin{proof}[Proof of Lemma \ref{Lemma:ZeroesMeromorphicFunctions}]
Assume that $F$ is a meromorphic function in $\C$ such that, for each $z_0\in\calU$, there exists $a_{z_0}\in\Z$ such that $z_0+a_{z_0}\omega$ is either a pole or a zero of $F$. Since the points $z_0+a_{z_0}\omega$ are pairwise distinct, the set $\{z_0+a_{z_0}\omega\mid z_0\in\calU\}$ is not countable, hence $F=0$.
 \hfill $\square$
 \end{proof}

\begin{lemma}\label{Lemma:sigma}
 There exists $\omega\in\Omega$ such that, for $\eta=\eta(\omega)$, we have
$\Real(\omega\eta)>0$.
 \end{lemma}

\begin{proof}[Proof of Lemma \ref{Lemma:sigma}]
Recall that the Weierstrass sigma function is entire and not zero, and has a zero at each point of $\Omega$. Using either Schwarz Lemma or the canonical product \eqref{Equation:sigmaDef}, we deduce that it has order of growth (see for instance \cite[Chap.2 \S1]{Lang}) at least $2$ (as a matter of fact the order is $2$).

If we had $\Real(\omega\eta)\leqslant 0$ for all $\omega\in\Omega$, we would deduce from 
\eqref{Equation:sigma} that there exists a constant $C>0$ such that, for all $z\in\C$, 
\[
|\sigma(z)|\leqslant C\exp\bigl( |z| \max\{\Real(\eta_1),\Real(\eta_2)\} \bigr)
\]
with 
\[
C=\sup\{|\sigma(x\omega_1+y\omega_2)|\; \mid \; 0\leqslant x,y\leqslant 1\},
\]
and $\sigma$ would have order $\leqslant 1$.
 \hfill $\square$
 \end{proof}

\begin{proof}[Proof of Proposition \ref{Proposition:sigma}]
 \null\hfill\break
 1. 
Thanks to Lemma \ref{Lemma:sigma}, there is no loss of generality in assuming $\Real(\omega_1\eta_1)>0$. Suppose that there exists a nonzero polynomial in $4+r+s $ variables $P\in\C[T_0,T_1,T_2,T_3,X_1,\dots,X_{r+s}]$ of degree $L_0\geqslant 1$ in $T_0$ such that the meromorphic function 
\[
F(z)=P\bigl(\sigma(z), z, \wp(z), \zeta(z), \rme^{t_1z},\dots, \rme^{t_rz}, f_{u_1}(z),\dots,f_{u_s}(z)\bigr)
\]
is zero. Write 
\[
P(T_1,T_2,T_3,X_1,\dots,X_{r+s},T_0)=
P_0(T_1,T_2,T_3,X_1,\dots,X_{r+s})T_0^{L_0}+
P_1(T_0,T_1,T_2,T_3,X_1,\dots,X_{r+s} ),
\]
with $P_0(T_1,T_2,T_3,X_1,\dots,X_{r+s})\not=0$ while $P_1(T_0,T_1,T_2,T_3,X_1,\dots,X_{r+s})$ has degree $<L_0$ in $T_0$ , so that 
\begin{equation}\label{Equation=F0+F1}
F(z)=F_0(z)\sigma(z)^{L_0}+F_1(z)
\end{equation}
with 
\begin{equation}
\label{F_0etF_1}
\begin{gathered}
F_0(z)=P_0\bigl(z, \wp(z), \zeta(z), \rme^{t_1z},\dots, \rme^{t_rz}, f_{u_1}(z),\dots,f_{u_s}(z) \bigr),
\\
F_1(z)=P_0\bigl(\sigma(z), z, \wp(z), \zeta(z), \rme^{t_1z},\dots, \rme^{t_rz}, f_{u_1}(z),\dots,f_{u_s}(z)\bigr).
\end{gathered}
\end{equation} 
Hence $F_0$ is a nonzero meromorphic function. 
 
 We use Lemma \ref{Lemma:ZeroesMeromorphicFunctions} with $F=F_0$ and $\omega=\omega_1$:
 let $z_0\in\C\smallsetminus\Omega$ be such that, for all $a\in\Z$, $z_0+a\omega_1$ is not a pole of $F_0$ and $F_0(z_0+a\omega_1)\not=0$. 
 
Using \eqref{Equation:sigma}, we deduce that there exists a constant $c_1(z_0)=c_1>1$ such that, for any $a\in\Z$ with sufficiently large $|a|$,
\begin{equation}\label{Equation:BorneF0}
|\sigma(z_0+a\omega_1)| \geqslant \rme^{|a|^2\Real(\omega_1\eta_1)}c_1^{-|a|}.
\end{equation}
From \eqref{Equation:sigma} and \eqref{Equation:lambda} it follows that there exists a constant $c_2(z_0)=c_2>1$ such that, for sufficiently large $|a|$, 
\begin{equation}\label{Equation:BorneF1}
|F_1(z_0+a\omega_1)|\leqslant \rme^{|a|^2\Real(\omega_1\eta_1)(L_0-1)}c_2^{|a|}.
\end{equation}
Combining \eqref{Equation=F0+F1}, \eqref{Equation:BorneF0} and \eqref{Equation:BorneF1} with the assumption $F=0$, we deduce that there exists a constant $c_3(z_0)=c_3>1$ such that, for sufficiently large $|a|$, 
 \[
 |F_0(z_0+a\omega_1)|\leqslant \rme^{-|a|^2\Real(\omega_1\eta_1) }c_3^{|a|}.
 \]
Write 
\begin{equation}\label{Equation:P0}
P_0(T_1,T_2,T_3,X_1,\dots,X_{r+s})=\sum_{\um}\sum_{\uell} p_{\um,\uell} T_1^{m_1}T_2^{m_2}T_3^{m_3}X_1^{\ell_1}\cdots X_{r+s}^{\ell_{r+s}},
\end{equation}
where $\um,\uell$ stands for $(m_1,m_2,m_3,\ell_1,\dots,\ell_{r+s})$. Then $F_0(z_0+a\omega_1)=Q(a, v_1^a,\dots,v_{r+s}^a)$ where
\begin{equation}\label{Equation:Q0}
\begin{aligned}
Q(X_0,X_1,\dots,X_{r+s})=
\sum_{\um}\sum_{\uell} p_{\um,\uell} 
&\rme^{\ell_1t_1z_0} \cdots \rme^{\ell_rt_rz_0}
\bigl(f_{u_1}(z_0)\bigr)^{\ell_{r+1}}\cdots\bigl(f_{u_s}(z_0)\bigr)^{\ell_{r+s}}
\\
&(z_0+X_0\omega_1)^{m_1}\wp(z_0)^{m_2}(\zeta(z_0)+X_0\eta_1)^{m_3}
X_1^{\ell_1}\cdots X_{r+s}^{\ell_{r+s}}
\end{aligned}
\end{equation}
and 
\begin{equation}\label{Equation:vj}
v_i=\rme^{t_i\omega_1}, \quad (1\leqslant i\leqslant r), \qquad 
v_{r+j}=\lambda(u_j,\omega_1), \quad (1\leqslant j\leqslant s).
\end{equation}
Lemma \ref{Lemma:vanderMonde1variable2} implies that $F_0$ vanishes at $z_0+a\omega_1$ for all sufficiently large $|a|$, which is a contradiction.
 
\medskip\noindent
2. The proof of the second part of Proposition \ref{Proposition:sigma} is the same. The meromorphic function $F$ and the polynomial $P$ are now 
\[
F(z)=P\bigl(\sigma(z), z, \wp(z), \zeta(z), \rme^{t_1z},\dots, \rme^{t_rz}, f_{u_1}(z),\dots,f_{u_s}(z),
f_{u_1}(\alpha_1z),\dots,f_{u_s}(\alpha_sz), 
\bigr)
\]
and
\[
\begin{aligned}
P(T_0,T_1,T_2,T_3,X_1,\dots,X_{r+2s})=&
P_0(T_1,T_2,T_3,X_1,\dots,X_{r+2s})T_0^{L_0}
\\
&
+
P_1(T_0,T_1,T_2,T_3,X_1,\dots,X_{r+2s}).
\end{aligned}
\] 
We replace the definition \eqref{F_0etF_1} of $F_0$ and $F_1$ by setting
\[
\begin{gathered}
F_0(z)=P_0\bigl(z, \wp(z), \zeta(z), \rme^{t_1z},\dots, \rme^{t_rz}, f_{u_1}(z),\dots,f_{u_s}(z),
f_{u_1}(\alpha_1z),\dots,f_{u_s}(\alpha_sz) \bigr),
\\
F_1(z)=P_0\bigl(\sigma(z), z, \wp(z), \zeta(z), \rme^{t_1z},\dots, \rme^{t_rz}, f_{u_1}(z),\dots,f_{u_s}(z),
f_{u_1}(\alpha_1z),\dots,f_{u_s}(\alpha_sz)\bigr)
\end{gathered}
\]
and the definition \eqref{Equation:P0} of $P_0$ with
\[
P_0(T_1,T_2,T_3,X_1,\dots,X_{r+2s})=\sum_{\um}\sum_{\uell} p_{\um,\uell} T_1^{m_1}T_2^{m_2}T_3^{m_3}X_1^{\ell_1}\cdots 
X_{r+2s}^{\ell_{r+2s}}.
\]
Formula \eqref{Equation:Q0} becomes
\[
\begin{aligned}
Q(X_0,X_1,\dots,X_{r+2s})&=
\sum_{\um}\sum_{\uell} p_{\um,\uell} 
\rme^{\ell_1t_1z_0} \cdots \rme^{\ell_rt_rz_0}\\
&
\bigl(f_{u_1}(z_0)\bigr)^{\ell_{r+1}}\cdots\bigl(f_{u_s}(z_0)\bigr)^{\ell_{r+s}}
\bigl(f_{u_1}(\alpha_1z_0)\bigr)^{\ell_{r+s+1}}\cdots\bigl(f_{u_s}(\alpha_sz_0)\bigr)^{\ell_{r+2s}}
\\
&(z_0+X_0\omega_1)^{m_1}\wp(z_0)^{m_2}(\zeta(z_0)+X_0\eta_1)^{m_3}
X_1^{\ell_1}\cdots 
X_{r+2s}^{\ell_{r+2s}}.
\end{aligned}
\]
Finally formula \eqref{Equation:vj} is replaced with
\[
v_i=\rme^{t_i\omega_1}, \quad (1\leqslant i\leqslant r), \qquad 
v_{r+j}=\lambda(u_j,\omega_1),\quad
v_{r+s+j}=\lambda(u_j,\alpha_j\omega_1) \qquad (1\leqslant j\leqslant s).
\]
 \hfill $\square$
 \end{proof}
 
 \begin{proof}[Proof of Theorem \ref{Theorem:AlgebraicIndependenceFunctions}]
 Part 1 of Theorem \ref{Theorem:AlgebraicIndependenceFunctions} follows from Propositions \ref{Proposition:sansCMnisigma} and \ref{Proposition:sigma}.1.
 
Part 2 of Theorem \ref{Theorem:AlgebraicIndependenceFunctions} follows from Propositions \ref{Proposition:independance_f_uCM} and \ref{Proposition:sigma}.2. 
 \hfill $\square$
 \end{proof}

\section{Appendix: Weierstrass vs Riemann}

This section is motivated by some remarks by David Masser in his book \cite[Chap.20]{M}, in particular Lemma 20.7 and the two following exercises:

\medskip
\noindent
{\bf 
Exercise 20.87. p.274 } {\em Show that the Riemann zeta function and the Weierstrass zeta function are algebraically independent over $\C$.}
[{\em Hint: } the first is a function of $s$ and the second a function of $z$.]

\medskip
\noindent
{\bf 
Exercise 20.88. p.275} {\em No, seriously, show that the Riemann zeta function $\zeta(z)$ and the Weierstrass zeta function $\zeta(z)$ are algebraically independent over $\C$. Well, you know what I mean\dots}

\bigskip
Here, in order to avoid any confusion, we denote by $\zeta_R$ the Riemann zeta function and we keep the letter $\zeta$ for the Weierstrass zeta function associated with the elliptic function $\wp$. 

\begin{proposition} \label{prop:WeierstrasVsRiemann}
	The four functions $z,\wp(z),\zeta(z), \zeta_R(z)$ are algebraically independent over $\C$. 
\end{proposition}

\begin{proof}
From $\zeta(3)=1.202\dots<5/4$ we deduce 
\[
\left(\frac 2 3 \right)^\sigma+\left(\frac 2 4 \right)^\sigma+\cdots+\left(\frac 2 n \right)^\sigma+\cdots<1
\]
for $\sigma=3$, hence for $\sigma\geqslant 3$. 
Therefore, with the classical notations $s=\sigma+\rmi t$ where $\sigma$ and $t$ are real, we have
\begin{equation}\label{EquationZetaRiemann}
|\zeta_R(s)|< 2^{1-\sigma}
\end{equation}
for $\sigma\geqslant 3$.

	We select a basis $\omega_1,\omega_2$ of the lattice $\Omega$ with positive real parts of $\omega_1$ and $\omega_2$. Assume that there is a nonzero polynomial $A(X_0,X_1,X_2,Y)\in\Z[X_0,X_1,X_2,Y]$ such that $A\bigl(z,\wp(z),\zeta(z),\zeta_R(z)\bigr)=0$. Without loss of generality we may assume that $A$ is not divisible by $Y-1$: writing 
\[ 
	A(X_0,X_1,X_2,Y)=A_0(X_0,X_1,X_2)+(Y-1)A_1(X_0,X_1,X_2,Y)
\] 
we assume $A_0\not=0$. Fix $v\in\C \smallsetminus \Omega$ with real part $\geqslant 1$. We have
\[ 
	A\bigl(v+a\omega_1+b\omega_2,\wp(v),\zeta(v)+a\eta_1+b\eta_2, \zeta_R(v+a\omega_1+b\omega_2)\bigr)=0
\] 
	for all $(a,b)\in\Z$ with $a>0$, $b>0$. For $\max\{a,b\}$ sufficiently large, we deduce from \eqref{EquationZetaRiemann}
\[ 
	|\zeta_R(v+a\omega_1+b\omega_2)-1|\leqslant 
	|2^{1-v-a\omega_1-b\omega_2}|\leqslant 
	 2^{-c\max\{a,b\} }
\] 
	with $c=\min\{\Re(\omega_1),\Re(\omega_2)\}>0$. 
	Hence 
\[ 
	(\zeta_R(v+a\omega_1+b\omega_2)-1) A_1\bigl(v+a\omega_1+b\omega_2,\wp(v),\zeta(v)+a\eta_1+b\eta_2, \zeta_R(v+a\omega_1+b\omega_2)\bigr) 
\] 
	tends to $0$ as $\max\{a,b\}\to \infty$. Consequently 
\[ 
	|A_0(v+a\omega_1+b\omega_2,\wp(v),\zeta(v)+a\eta_1+b\eta_2) |
\] 
	also tends to $0$ as $\max\{a,b\}\to \infty$. Let $b$ be a positive integer. Consider the polynomial 
\[
	P_b(X)=A_0(v+X\omega_1+b\omega_2,\wp(v),\zeta(v)+X\eta_1+b\eta_2) \in\C[X].
\]
Since $|P_b(a)|$ tends to $0$ as $a\to\infty$, this polynomial is $0$. Let $z\in\C$ and let 
\[ 
Q_z(Y)=A_0(v+z\omega_1+Y\omega_2,\wp(v),\zeta(v)+z\eta_1+Y\eta_2) \in\C[Y],
\]
so that $Q_z(b)=P_b(z)$. 
Since $P_b(z)=0$ for all positive integers $b$, the polynomial $Q_z(Y)$ has infinitely many zeroes, hence it is $0$. 
We deduce 
\[ 
	A_0(v+X\omega_1+Y\omega_2,\wp(v),\zeta(v)+X\eta_1+Y\eta_2) =0.
\] 
 Since $\omega_1\eta_2-\omega_2\eta_1\not=0$, we deduce that the polynomial $A_0(X,\wp(v),Z)\in\C[X,Z]$ is $0$ for all $v\in\C\smallsetminus\Omega$, a contradiction.
 \hfill $\square$
 \end{proof}

Using Propositions \ref{Prop:SSPIF} and \ref{prop:WeierstrasVsRiemann}, we deduce:

\begin{corollary} 
	For almost all $n$--tuples $(z_1,\dots,z_n)$ of complex numbers, the $4n$ numbers 
\[ 
	z_1,\dots,z_n,
	\wp(z_1),\dots,\wp(z_n),
	\zeta(z_1),\dots,\zeta(z_n),
	\zeta_R(z_1),\dots,\zeta_R(z_n)
\] 
	are algebraically independent over $\Q(g_2,g_3)$. 
\end{corollary}

\bibliographystyle{plain}

\authoraddresses{
Michel Waldschmidt\\ 
Sorbonne Universit\'{e}, CNRS, IMJ-PRG, F-75005 Paris, France
\\	
\email michel.waldschmidt@imj-prg.fr
}

\end{document}